\documentclass[11pt,leqno]{article}
\usepackage{graphicx}
\usepackage{amsfonts, amsthm}
\usepackage{amsxtra}
\usepackage[latin1]{inputenc}
\usepackage{fontenc}
\usepackage[dvips]{epsfig}
\begin{document}
\newtheorem{theo}{Theorem}[section]
\newtheorem{defi}[theo]{Definition}
\newtheorem{lemm}[theo]{Lemma}
\newtheorem{prop}[theo]{Proposition}
\newtheorem{rem}[theo]{Remark}
\newtheorem{exam}[theo]{Example}
\newtheorem{cor}[theo]{Corollary}
\def\Z{\mathbb{Z}} 
\def\R{\mathcal{R}} 
\def\I{\mathcal{I}} 
\def\C{\mathbb{C}}   
\def\N{\mathbb{N}}   
\def\PP{\mathbb{P}}   
\def\Q{\mathbb{Q}}   
\def\L{\mathcal{L}}   
\def\ol{\overline}   
\def\bs{\backslash}  
\def\part{P} 

\newcommand{\gcD}{\mathrm {\ gcd}}  
\newcommand{\End}{\mathrm {End}}  
\newcommand{\Aut}{\mathrm {Aut}}  
\newcommand{\GL}{\mathrm {GL}}  
\newcommand{\SL}{\mathrm {SL}}  
\newcommand{\PSL}{\mathrm {PSL}}  
\newcommand{\Mat}{\mathrm {Mat}}  

\newcommand\ga[1]{\overline{\Gamma}_0(#1)}      
\newcommand\pro[1]{\mathbb{P}^1(\mathbb{Z}_{#1})}  
\newcommand\Zn[1]{\mathbb{Z}_{#1}}
\newcommand\equi[1]{\stackrel{#1}{\equiv}}
\newcommand\pai[2]{[#1:#2]}    
\newcommand\modulo[2]{[#1]_#2} 
\newcommand\sah[1]{\lceil#1\rceil}       
\def\sol{\phi} 
\begin{center}
{\LARGE\bf Singular measures of circle homeomorphisms with two breakpoints}
 \\
\vspace{.25in} {\large {\sc Akhtam Dzhalilov\footnote{Faculty of Mathematics, Samarkand State University, Boulevard st. 15, 703004 Samarkand, Uzbekistan. \quad E-mail: a\_dzhalilov@yahoo.com},
Isabelle Liousse \footnote{Laboratoire Paul Painlev\'{e}, Universit\'{e} Lille I, F-59655 Villeneuve d'Ascq, France. 
\quad E-mail: isabelle.liousse@univ-lille1.fr},
Dieter Mayer\footnote{Institut f\"ur Theoretische Physik, TU Clausthal, D-38678 Clausthal-Zellerfeld, Germany. 
\quad E-mail: dieter.mayer@tu-clausthal.de}}}

\end{center}


\title{Singular measures of circle homeomorphisms\\
with two break points}

\begin{abstract}
Let $T_{f}$ be a circle homeomorphism with two break points $a_{b},c_{b}$ and 
irrational rotation number $\varrho_{f}$.
Suppose that the derivative $Df$ of  its lift $f$ is absolutely continuous 
on every connected interval of the set  $S^{1}\backslash\{a_{b},c_{b}\}$, that
$DlogDf \in L^{1}$ and the product of the jump ratios of $ Df $ 
at the break points is nontrivial, i.e. 
$\frac{Df_{-}(a_{b})}{Df_{+}(a_{b})}\frac{Df_{-}(c_{b})}{Df_{+}(c_{b})}\neq1$.
We prove that the unique $T_{f}$- invariant probability measure $\mu_{f}$ is then singular with respect to Lebesgue measure $l$ on $S^{1}$.
\end{abstract}

\section{Introduction}
Circle homeomorphisms constitute one important class of one-dimentional dynamical systems.The investigation of  their properties was initiated by Poincar\'{e} \cite{Po1885}, who came across them in his studies of differential equations more than a century ago. Since then interest in these maps never diminished. Circle maps are also important because of their applications to natural sciences (see for instance \cite{Cv1989}).
Let $T_{f}$ be an orientation preserving circle homeomorphism with lift $f:\mathbb{R}\rightarrow \mathbb{R}$, $f$  continuous, strictly increasing and $f(x+1)=f(x)+1$, $x\in \mathbb{R}$.
We identify the unit circle  $S^{1}=\mathbb{R}/ \mathbb{Z}$ with the half open interval $[0,1)$.
 The circle homeomorphism $T_{f}$ is then defined by $T_{f}x=f(x)$ mod 1, $x\in S^{1}$.
An important conjugacy invariant characteristic of orientation preserving homeomorphisms is the rotation number $\varrho(f)$. 
If $T_{f}$ is  a circle homeomorphism with lift $f$, then the rotation number $\varrho=\varrho(f)$ is defined by
$$\rho(f)=\left( \overset{}{\underset{n\rightarrow \infty}{\lim
}}\frac{f^{n}(x)}{n}\right)  \text{mod}\ 1,$$
with $f^{n}$ the $n$-th  iterate  of $f$.
This limit exists and is independent of the choice of the lift and the point $x\in \mathbb{R}$. 
If $\varrho$ is irrational, then for sufficiently smooth diffeomorphisms the trajectory of an arbitrary point is dense on the circle, and the diffeomorphism itself can be reduced to the pure rotation $T_{\varrho}x=(x+\varrho)  $ mod 1
by an angle $\varrho$ through a change of coordinates. This  result was proved by Denjoy \cite{De1932}.
More precisely, Denjoy proved that if  $f\in C^{1}(R^{1})$ and $var(\text{log}Df)<\infty$,
then there exists a circle homeomorphism  $T_{\varphi}$ such that 
\begin{equation}\label{eq1}
T_{f}\circ T_{\varphi}=T_{\varphi}\circ T_{\varrho}.
\end{equation}
It is a well known fact that a circle homeomorhism $T_{f}$ with irrational rotation number $\varrho$ is 
strictly ergodic i.e. admits an unique $T_{f}$-invariant probability measure $\mu_{f}$. Note, that the conjugating map $T_{\varphi}$ and the invariant measure $\mu_{f}$ are related by $T_{\varphi}x=\mu_{f}([0,x])$ (see \cite{CFS1982}).
This last relation implies that regularity properties of the conjugating map  $T_{\varphi}$ are closely related to the  existence of an absolutely continuous invariant measure $\mu_{f}$ with a regular density.

The problem of smoothness of the conjugacy of smooth diffeomorphisms is now very well understood(see for instance \cite{Ar1961, Mo1966, He1979, KO1989, KS1989, Yo1984}).

An important result is the one by M. Herman \cite{He1979}:

\begin{theo}\label{Theorem 1.} If $T_{f}$ is a $C^{2}$-diffeomorphism with rotation number $\varrho=\varrho(f)$ 
of bounded type (that means the entries in the continued fraction expansion of $\varrho$ are bounded)
and $T_{f}$ is close to $T_{\varrho}$ then $\mu_{f}$ is absolutely continuous with respect to Lebesgue measure.
\end{theo}

Katznelson-Ornstein \cite{KO1989} and Khanin-Sinai \cite{KS1989} gave new proofs and an improved global version of this theorem 
in showing that it is not necessary to assume that $T_{f}$ is close to $T_{\varrho}$:

\begin{theo}(Katznelson-Ornstein)\label{KO}. Let $T_{f}$ be an orientation
preserving $C^{1}$-circle diffeomorphism. If $f$ is absolutely
continuous, $D(logDf)\in L^{p}$ for some $p>1$ and the 
rotation number $\rho=\rho(f)$ is of bounded type, 
 then the invariant measure $\mu(f)$ is absolutely continuous
with respect to Lebesgue measure.
\end{theo}
The  result proved by Khanin and Sinai in \cite{KS1989} is the following:
\begin{theo}(Khanin-Sinai)\label{KS}. Let $T_{f}$ be a $ C^{2+\varepsilon}$ circle diffeomorphism  with $\varepsilon>0$,
and let the rotation number $\rho=\rho(f)$ be a Diophantine number
with exponent $\delta\in(0,\varepsilon)$, i.e., there is a constant $c(\varrho)$
such that
$$
|\rho-\frac{p}{q}|\geq\frac{c(\rho)}{q^{2+\delta}}\ \text{for any}\ \frac{p}{q}\in\mathbb{Q}.
$$
Then the conjugating map $T_{\varphi}$
belongs to $C^{1+\varepsilon-\delta}$. 
\end{theo}
Note, that the condition
$T_{f}\in C^{2+\varepsilon}$ is sharp, because there is a set of full Lebesgue measure in $[0,1]$ such that for any rotation number in
this set there are  $C^{2}$-diffeomorphisms for which the
conjugating map $T_{\varphi}$  is singular \cite{HS1982}.

An important and interesting class of circle homeomorphisms are homeomorphisms with singularities. The simplest among
them are critical circle homeomorphisms and homeomorphisms with break points. We call this latter class {\bf{$P$-homeomorphisms}}. In general their ergodic properties like the invariant measures , their renormalization and also their rigidity properties  are different from the properties of diffeomorphisms (see \cite{dMvS1993} chapter I and IV,  \cite{He1979} chapter VI, \cite{KhKm2003}, \cite{dFdM1999}).

The invariant measures of critical circle homeomorphisms, that means $C^{3}-$ smooth circle homeomorphisms with a finite number of critical points of polynomial type have been studied in (\cite{GrSw1993}):

\begin{theo} (Graczik-Swiatek)\label{GS}. Critical circle homeomorphisms with irrational rotation number 
have an invariant measure singular with respect to Lebesgue measure.
\end{theo}
The class of $P$-homeomorphisms consists of orientation preserving circle homeomorphisms $T_{f}$ 
which are differentiable away from countable many points, the so called break points,
at which left and right derivatives, denoted respectively by $Df_{-}$ and $Df_{+}$, exist such that
\begin{itemize}
\item[i)] there exist constants $0<c_{1}<c_{2}<\infty$ with $c_{1}<Df(x)<c_{2}$ for all  $x\in S^{1}\setminus BP(f)$,\\
$c_{1}<Df_{-}(x_{b})<c_{2}<$ \,\, and  $c_{1}<Df_{+}(x_{b})<c_{2}$ \,\, for all  $x_{b}\in BP(f)$, \, the set of break points of $f$;
\item[ii)]  $logDf$ has bounded variation.
\end{itemize}
In this case $logDf$,  $logDf_{-}$, $logDf_{+}$ and $logDf^{-1}$, $logDf_{-}^{-1}$, $logDf_{+}^{-1}$ all have the same total variation denoted by $v=Var(logDf)$.

The ratio $\sigma_{f}(x_{b})= \frac{Df_{-}({c_{b}})}{Df_{+}({c_{b}})}$ is
called the jump ratio of $T_{f}$ in $x_{b}$ .

Piecewise linear (PL) orientation preserving circle homeomorphisms with piecewise constant 
derivatives are the simplest examples of class $P$-homeomorphisms. They occur in many other areas
of mathematics such as group theory, homotopy theory and logic via the Thompson groups (see \cite{St1992}. 
PL-homeomorphisms were considered first by Herman in \cite{He1979} as examples of homeomorphisms of arbitrary irrational rotation number which admit no invariant $\sigma -$finite measure equivalent to  Lebesgue measure.
\begin{theo}(Herman)\label{H}. A PL-circle homeomorphisms with two break points and irrational rotation number
has an invariant measure absolutely continuous with respect to Lebesgue measure if and only if its break points belong to the same orbit.
\end{theo} 
General (non PL) class $P$-homeomorphisms with one break point have been studied by Dzhalilov and Khanin in \cite{DK1998}.
The character of their results for such circle maps is  quite different from the one for $C^{2+\varepsilon}$ diffeomorphisms.
The main result of \cite{DK1998} is the following:

\begin{theo}\label{DK}. Let $T_{f}$ be a class $P$-homeomorphism with one break point $c_{b}$. If the rotation number $\rho_{f}$ is irrational and $T_{f}\in C^{2+\varepsilon}(S\backslash\{c_{b}\})$  for some $\varepsilon>0$, then the $T_{f}$-invariant probability measure $\mu_{f}$ is singular with respect to Lebesgue measure l on $S^{1}$, i.e. there exists a measurable
 subset $A\subset S^{1}$  such that $\mu_{f}(A)=1$ and $l(A)=0$.\\
\end{theo}
I. Liousse proved in \cite{Li2005} the same result for "generic" PL-homeomorphisms with irrational rotation number of bounded type. In a next step Dzhalilov and I. Liousse studied in \cite{DL2006} circle homeomorphisms with two break points. Their result is the following:

\begin{theo}\label{DL}. Let $T_{f}$ be a class $P$-homeomorphism  satisfying the following conditions:
\begin{itemize}
\item[i)]   $T_{f}$ has irrational rotation number $\rho_{f}$ of bounded type;
\item[ii)]   there exists constants $k_{i}>0$ such that 
$|Df(x)-Df(y)|\le k_{i}|x-y|$  on every continuity interval of $Df$;
\item[iii)]  $T_{f}$ has two break points not on the same orbit of $T_{f}$.
\end{itemize}
Then the $T_{f}$- invariant probability measure $\mu_{f}$ is singular with respect to Lebesgue measure.
\end{theo}

In the present paper we continue our study of invariant measures for circle 
homeomorphisms  $T_{f}$ with two break points and arbitrary irrational rotation number $\rho_{f}$.
The main result of our paper is the following:
\begin{theo}\label{DLM}. Let $T_{f}$ be a class $P$-homeomorphism  satisfying the following conditions:
\begin{itemize}
\item[(a)]  the rotation number $\rho=\rho_{f}$ of $T_{f}$ is irrational;
\item[(b)] $T_{f}$ has two break points $a_{b}$, $c_{b}$ and the product of the jump ratios of $Df$ at the break points is nontrivial i.e. $\sigma_{f}(a_{b})\cdot\sigma_{f}(c_{b})\ne 1$.
\item[(c)]  $Df(x)$ is absolutely continuous on
every connected interval of $S^{1}\backslash\{a_{b}, c_{b}\}$ and the second derivative $D^{2}f\in L^{1}$;
\end{itemize}
Then the $T_{f}$- invariant probability  measure $\mu_{f}$ is singular with respect to Lebesgue measure.\end{theo}

\begin{rem} {\rm Obviously condition (c) is weaker than  a Lipschiz condition for $Df$. In the case  when  $T_{f}$  has two break points on the same orbit our Theorem \ref{DLM} gives a new proof of the result in  \cite{DK1998}, but with a weaker condition than $C^{2+\varepsilon}$.} on $T_{f}$. 
\end{rem}
A direct consequence of our Theorem \ref{DLM} is
\begin{theo}\label{DLM1} Let $T_{f}$ be a circle homeomorphism satisfying 
 condition (c) of  Theorem \ref{DLM} and the conditions
\begin{itemize}
\item[(a')] the rotation number $\rho_{f}$ of $T_{f}$ is irrational of bounded type;
\item[(b')]  $T_{f}$ has two break points $a_{b}$, $c_{b}$ with disjoint orbits and 
$\sigma(a_{b})\sigma(a_{b})=1.$
\end{itemize}
Then the $T_{f}-$invariant measure $\mu_{f}$ is singular with respect to Lebesgue measure.
\end{theo}

\section{Preliminaries and Notations}

Let $T_{f}$  be an orientation preserving homeomorphism of the circle with lift $f$ and irrational rotation number $\rho=\rho_{f}$.
  We take an arbitrary point $x_{0}\in S^{1}$ and
consider the trajectory of this point under the action of $T_{f}$,
i.e., the set of points $\{x_{i}=T_{f}^{i}x_{0}, \
i\in\mathbb{Z} \}$. According to a classical
 theorem of Poincar\'{e} (see \cite{CFS1982}), the order of the points
along the trajectory is the same as in the case of the linear
rotation $T_{\rho}$ of the circle, i.e. for the sequence
$\{\overline{x}_{i}=\{x_{0}+i\rho\}, mod 1, \, i\in\mathbb{Z} \}$. This
important property allows one to define a sequence  of natural
partitions of the circle closely related to the continued fraction expansion of the
number $\rho$ .

 We denote by  $\left\lbrace k_{n}, n \in \mathbb{N} \right\rbrace $ the sequence of entries in the continued fraction
expansion of $\rho$, so that $\varrho=\left[k_{1},k_{2},...,k_{n},...\right]=
\left( 1/k_{1}+\left( 1/\left( k_{2}+...+1/k_{n}+...\right) \right) \right)  $. 
For $n \in \mathbb{N}$ we denote by $p_{n}/q_{n}=[k_{1},k_{2},...,k_{n}] $ the convergents of $\rho$. 
Their denominators $q_{n}$ satisfy the recursion relation 
$q_{n+1}=k_{n+1}q_{n}+q_{n-1}, \ n\geq 1, \ q_{0}=1, \ q_{1}=k_{1}$.

For an arbitrary point $x_{0}\in S^{1}$ denote by
$\Delta_{0}^{(n)}(x_{0})$ the closed interval with endpoints
$x_{0}$ and $x_{q_{n}}.$ For $n$ odd $x_{q_{n}}$ is to the left of $x_{0}$, for $n$ even it is to the right. Denote by $\Delta_{i}^{(n)}(x_{0})$ the iterates of the
interval $\Delta_{0}^{(n)}(x_{0})$ under $T_{f}$:
$\Delta_{i}^{(n)}(x_{0})=T^{i}_f\Delta_{0}^{(n)}(x_{0}),i\ge1$.

It is well known (since Denjoy) that the system of intervals  
\begin{equation}\label{eq2}
\xi_{n}(x_{0})=\left\lbrace \Delta_{i}^{(n-1)}(x_{0}), \  0\leq i<q_{n}; \    \Delta_{j}^{(n)}(x_{0}), \ 0\leq j<q_{n-1}\right\rbrace 
\end{equation}
cover the whole circle and that their interiors are mutually disjoint.
The partition $\xi_{n}(x_{0})$ is called the $n$-th \textbf{dynamical partition} of the point 
$x_{0}$ with \textbf{generators} $\Delta_{0}^{(n-1)}(x_{0})$ and  $\Delta_{0}^{(n)}(x_{0})$. We briefly recall the structure of the dynamical partitions. The  passage from $\xi_{n}(x_{0})$ to $\xi_{n+1}(x_{0})$ is simple:
all intervals of order $n$ are preserved and each of the intervals $ \Delta_{i}^{(n-1)}(x_{0}), \  0\leq
i<q_{n},$ is divided into $k_{n+1}+1$ intervals:
\begin{eqnarray*}
\Delta_{i}^{(n-1)}(x_{0})=\Delta_{i}^{(n+1)}(x_{0})\cup\bigcup_{s=0}^{k_{n+1}-1}\Delta_{i+q_{n-1}+sq_{n}}^{(n)}(x_{0}).
\end{eqnarray*}
The following Lemma plays a key role for studying metrical properties of the homeomorphism $T_{f}$.
\begin{lemm}\label{lem2.1} Let $T_{f}$ be a $P$- circle homeomorphism with a finite number
 of break points $z^{(i)},$ $i=1,2,...,m$ and irrational rotation number $\rho_{f}$.
If $x_{0}\in S^{1},$ $ n\ge 1$ 
and  $z^{(i)}\notin\left\lbrace  T^{i}x_{0}, 0\leq
i<q_{n},\right\rbrace $ then 
\begin{equation}\label{eq3} 
e^{-v}\leq \prod_{i=0}^{q_{n}-1}DT_{f}^{i}(x_{0})\leq e^{v},
\end{equation}
where $v=Var(logDf)$.
\end{lemm}

Inequality (\ref{eq3}) is called the \textbf{Denjoy inequality}.
The proof of Lemma \ref{lem2.1}  is just like in the case of  diffeomorphisms (see for instance \cite{KS1989}).
Using  Lemma \ref{lem2.1} it can be shown easily that the lenghts of the intervals of the dynamical partition $\xi_{n}$ in (\ref{eq2}) are exponentially small:

\begin{cor}\label{cor2} Let $\Delta^{(n)}$ be an arbitrary element
of the dynamical partition $\xi_{n}(x_{0})$. Then
\begin{equation}\label{eq4}
l(\Delta^{(n)}) \leq const \,\, \lambda^{n},
\end{equation}
where $\lambda=(1+e^{-v})^{-1/2}<1.$
\end{cor}

\begin{defi}\label{defi3} Two homeomorphisms $T_{1}$ and $T_{2}$ of
the circle are said to be topologically equivalent if there exists
a homeomorphism $T_{\varphi}:S^{1}\rightarrow S^{1}$ such that
$T_{\varphi}(T_{1}(x))=T_{2}(T_{\varphi(x)})$ for any  $x\in S^{1}$.We call the homeomorphism $T_{\varphi}$ a conjugating map.
\end{defi}
From Corollary \ref{cor2} it follows that the trajectory of each point  is dense in $S^{1}$. 
This together with the monotonicity of the 
homeomorphism $T_{f}$ implies the following generalization of the classical Denjoy theorem. \\

\begin{theo}\label{Denjoy} Suppose that a homeomorphism $T_f$ with
an irrational rotation number $\rho$ satisfies the conditions of Lemma \ref{lem2.1}.Then $T_f$ is topologically equivalent to
the linear rotation $T_{\rho}$.
\end{theo}

\begin{defi}(see \cite{KO1989})\label{defKO} An interval $I=(\tau,t)\subset S^1$ is $q_{n}$-small
and its endpoints $\tau,t$ are $q_{n}$-close if the system of
intervals $T_{f}^{i}(I),\ 0\leq i<q_{n}$ are disjoint.
\end{defi}

It is known that the interval $(\tau,t)$ is $q_{n}$-small
if, depending on the parity of n, either $t\leq\tau\leq
T_{f}^{q_{n-1}}(t)\ \text{or} \ T_{f}^{q_{n-1}}(\tau)\leq
t\leq\tau.$\\

\begin{defi}\label{Def2.3} {\rm 
Let $C>1$ . We call two intervals of $S^{1}$ \textbf{C-comparable } if the ratio of their lengts is in $[C^{-1}, C].$}
\end{defi}
Lemma \ref{lem2.1} then implies
\begin{cor}\label{cor2.2} Suppose that a homeomorphism $T_f$ with
an irrational rotation number $\rho$ satisfies the conditions of 
Lemma \ref{lem2.1}. Then for any interval $I\subset S^{1}$ the intervals
$I$ and $T_{f}^{q_{n}}I$ are $e^{v}$-comparable. If the interval $I$ is
$q_{n}-$small then $l(T_{f}^{i}I)< const \, \lambda^{n} $ for all $i=1,2, ...,q_{n}-1.$
\end{cor}

\begin{lemm}\label{lemm2.2} Suppose, that a homeomorphism $T_f$ with
an irrational rotation number $\rho$ satisfies the conditions of 
Lemma \ref{lem2.1} and $x,y\in S^{1}$ are $q_{n}$-close. Then for any $0\leq l \leq q_{n}$ the following
inequality holds:
\begin{equation}\label{eq5}
e^{-v}\leq \frac {Df^l(x)}{Df^l(y)}\leq e^{v}.
\end{equation}
\end{lemm}
\begin{proof} Take any two $q_{n}$-close points $x,y\in S^{1}$ 
and $0\leq l\leq\ q_{n}-1$.
Denote by $I$ the open interval with endpoints $x$ and $y$.
Because the intervals $T_{f}^{i}(I),\ 0\leq i<q_{n}$ are disjoint, we obtain
\begin{displaymath}
|log Df^{q_n}(x)-log Df^{q_n}(y)|\le \sum_{s=0}^{q_{n}-1}|log f(T_{f}^{s}x)-log f(T_{f}^{s}y)|\le v,
\end{displaymath}
from which inequality (\ref{eq5}) follows immediately.
\end{proof}
Note that P-homeomorphisms $T_f$ satisfying the conditions of Lemma \ref{lem2.1} are ergodic with respect to Lebesgue measure , i.e. every $T_f$-invariant set has measure zero or one.

\section{Cross-ratio tools}

\begin{defi}\label{crossratio} {\rm The \textbf{cross-ratio } $Cr(z_1,z_2,z_3,z_4)$ of four points $z_i\in \mathbb{R}^{1}$, $i=1,2,3,4$, $z_1<z_2<z_3<z_4$ is defined as
$$
Cr(z_1,z_2,z_3,z_4)=\frac{(z_2-z_1)(z_4-z_3)}{(z_3-z_1)(z_4-z_2)}.
$$}
\end{defi}
\begin{defi}\label{distortion} {\rm
The \textbf{cross-ratio distortion} $Dist(z_1,z_2,z_3,z_4;f)$ of four points $z_i\in \mathbb{R}^{1}$, $i=1,2,3,4$, $z_1<z_2<z_3<z_4$ with respect to a strictly increasing
function $f$ on $ \mathbb{R}$ is defined as
$$
Dist(z_1,z_2,z_3,z_4;f)=\frac{Cr(f(z_1),f(z_2),f(z_3),f(z_4))}{Cr(z_1,z_2,z_3,z_4)}.
$$}
\end{defi}

Consider then a function $f:[a,b]\rightarrow R^{1}$, $[a,b]\subset S^{1}$ satisfying the following conditions:
\begin{enumerate}
\item[{\rm(i)}]
 $f \in C^{1}([a,b])$, $Df(x)\ge const >0$, $\forall x\in [a,b]$;
\item[{\rm(ii)}]
  $D^{2}f\in L^{1}([a,b])$.
\end{enumerate}

Fix an arbitrary $\varepsilon >0$. Since $D^{2}f\in L^{1}([a,b])$, it can be written in the form
\begin{equation}\label{eq6}
D^{2}f(x)=g_{\varepsilon}(x)+\theta_{\varepsilon}(x),\ x\in [a,b],
\end{equation}
where $g_{\varepsilon}$ is a continuous function on $[a,b]$ and $\Vert\theta_{\varepsilon}\Vert_{L^{1}}<\varepsilon$.

\begin{theo}\label{theo3.1}
 Suppose,  the function $f=f(x)$ satisfies the above
conditions ${i}),{ii})$. For $z_i\in [a,b]$, \ $i=1,2,3,4$, with 
$z_1<z_2<z_3<z_4$, the following estimate holds:

\begin{eqnarray*}
\label{eq6a}
|Dist(z_1,z_2,z_3,z_4;f)-1|&\leq& C_1|z_4-z_1|\underset{x,t\in
[a,b]}{\max}|g_{\varepsilon}(x)-g_{\varepsilon}(t)|+\notag \\
&+&C_1\overset{z_1}{\underset{z_2}{\int}}|\theta_{\varepsilon}(y)|dy+
C_{1}\Big(\overset{z_2}{\underset{z_1}{\int}}|D^{2}f(y)|dy\Big)^{2}
\end{eqnarray*}

where the constant $C_{1}>0$ depends only on the function $f$.\\
\end{theo}
\begin{proof} 
 Take  $z_i\in [a,b]\subset S^{1}$, \ $i=1,2,3,4$, with
$z_1<z_2<z_3<z_4$. The following equalities are easy to check:
\begin{eqnarray*}
f(z_k)-f(z_1)=Df(z_1)(z_k-z_1)+\overset{z_k}{\underset{z_1}{\int}}D^{2}f(y)(z_k-y)dy,\
k=2,3;\\
f(z_4)-f(z_l)=Df(z_4)(z_4-z_l)-\overset{z_4}{\underset{z_l}{\int}}D^{2}f(y)(y-z_l)dy,\
l=2,3.
\end{eqnarray*}
Using these relations we obtain:
\begin{eqnarray}\label{eq8}
&&Cr(f(z_1),f(z_2),f(z_3),f(z_4))=
\frac{f(z_2)-f(z_1)}{f(z_3)-f(z_1)}\cdot \frac{f(z_4)-f(z_3)}{f(z_4)-f(z_2)}=  \notag \\
&=& Cr(z_1,z_2,z_3,z_4)\cdot\frac{1+\frac{1}{Df(z_1)(z_2-z_1)}
\overset{z_2}{\underset{z_1}{\int}}D^{2}f(y)(z_2-y)dy}{1+\frac{1}{Df(z_1)(z_3-z_1)}
\overset{z_3}{\underset{z_1}{\int}}D^{2}f(y)(z_3-y)dy}\times  \\
&\times& \frac{1-\frac{1}{Df(z_4)(z_4-z_3)}
\overset{z_4}{\underset{z_3}{\int}}D^{2}f(y)(y-z_3)dy}{1-\frac{1}{Df(z_4)(z_4-z_2)}
\overset{z_4}{\underset{z_2}{\int}}D^{2}f(y)(y-z_2)dy}.\notag 
\end{eqnarray}
Setting\\
$
A(a,b)=\frac{1}{Df(a)(b-a)}
\overset{b}{\underset{a}{\int}}D^{2}f(y)(b-y)dy,
$\,\,\,\,and 
$B(a,b)=\frac{1}{Df(b)(b-a)}
\overset{b}{\underset{a}{\int}}D^{2}f(y)(y-a)dy
$\\
we can hence rewrite  $Dist(z_1,z_2,z_3,z_4;f)$ in the following form:

\begin{eqnarray*}
Dist(z_1,z_2,z_3,z_4;f)&=&\frac{1+A(z_1,z_2)}{1+A(z_1,z_3)}\times\frac{1-B(z_3,z_4)}{1-B(z_2,z_4)}=  \\
&=&(1+A(z_1,z_2))\cdot(1-A(z_1,z_3)+O(A^2(z_1,z_3))\cdot(1-B(z_3,z_4))\times  \\
&\times&(1+B(z_2,z_4)+O(B^2(z_2,z_4))).
\end{eqnarray*}
Therefore
\begin{eqnarray}\label{eq9}
Dist(z_1,z_2,z_3,z_4;f)&=&1+A(z_1,z_2)-A(z_1,z_3)+B(z_2,z_4)-B(z_3,z_4)+\\
&+&O\left(\Big(\overset{z_3}{\underset{z_2}{\int}}|D^{2}f(y)|dy\Big)^2
\right)\notag.
\end{eqnarray}
Set
$M_{1}=0.5\Big(\overset{}{\underset{x\in
(z_{1},z_{4})}{\inf }} Df(x)\Big)^{-1}. $\\

 To continue the proof of Theorem \ref{theo3.1} we need the following\\

\begin{lemm}\label{lemm3.1} Assume, that the function $f$ satisfies the
conditions of Theorem \ref{theo3.1}. Then for any $a,b\in S^{1}$, $a<b$ the
following identities hold:
$$A(a,b)=\overset{b}{\underset{a}{\int}}\frac{D^{2}f(y)}{2Df(y)}dy+G_1(a,b),\,\,\,\,
B(a,b)=\overset{b}{\underset{a}{\int}}\frac{D^{2}f(y)}{2Df(y)}dy+G_2(a,b),$$
where
\begin{eqnarray}\label{eq10}
|G_i(a,b)|\leq M_1(b-a)\underset{x,t\in
[a,b]}{\max}|g_{\varepsilon}(x)-g_{\varepsilon}(t)|+ \notag\\
+ M_1
\overset{b}{\underset{a}{\int}}|\theta_{\varepsilon}(y)|dy+
2M_1^{2}\left(\overset{b}{\underset{a}{\int}}|D^{2}f(y)|dy\right)^{2},\
i=1,2.
\end{eqnarray}
\end{lemm}
\begin{proof} We prove only the identity for $A(a,b)$, the one 
for $B(a,b)$ is similar. Set
$$G_1(a,b)=A(a,b)-\overset{b}{\underset{a}{\int}}\frac{D^{2}f(y)}{2Df(y)}dy.$$\\
It is clear that
\begin{eqnarray*}
|G_1(a,b)|&\le& \Big|A(a,b)-\overset{b}{\underset{a}{\int}}\frac{D^{2}f(y)}{2Df(a)}dy\Big|+
\Big|\overset{b}{\underset{a}{\int}}\frac{D^{2}f(y)}{2Df(a)}dy-\overset{b}
{\underset{a}{\int}}\frac{D^{2}f(y)}{2Df(y)}dy\Big|=\\
&=&\Big|A(a,b)-\overset{b}{\underset{a}{\int}}\frac{D^{2}f(y)}{2Df(A)}dy\Big|+
\frac{1}{2}\Big|\overset{b}{\underset{a}{\int}}\frac{D^{2}f(y)}{Df(y)Df(a)}dy
\overset{y}{\underset{a}{\int}}D^{2}f(t)dt\Big|.
\end{eqnarray*}
using this and the bound $(Df(x))^{-1} \leq 2\, M_1$ we get :
\begin{eqnarray}\label{eq11}
|G_1(a,b)| \leq \Big|A(a,b)-\overset{b}{\underset{a}{\int}}\frac{D^{2}f(y)}{2Df(a)}dy\Big|+2M_{1}^{2}
\Big(\overset{b}{\underset{a}{\int}}|D^{2}f(y)|^{2}dy\Big).
\end{eqnarray}

To get finally the estimate (\ref{eq10}) for $G_1(a,b)$ it is sufficient 
to estimate the difference
$A(a,b)-\overset{b}{\underset{a}{\int}}\frac{D^{2}f(y)}{2Df(a)}dy$.
Using the definition of $A(a,b)$ and the decomposition (\ref{eq6}) 
we obtain:
\begin{eqnarray*}
&&\Big|A(a,b)-\overset{b}{\underset{a}{\int}}\frac{D^{2}f(y)}{2Df(a)}dy\Big|
=\Big|\frac{1}{Df(a)}\overset{b}{\underset{a}{\int}}
D^{2}f(y)\left(\frac{b-y}{b-a}-\frac{1}{2}\right)dy\Big|=\\
&=&\frac{1}{Df(a)}\Big|\overset{b}{\underset{a}{\int}}
\Big(g_{\varepsilon}(y)+\theta_{\varepsilon}(y))
\left(\frac{b-y}{b-a}-\frac{1}{2}\right)dy\Big|
\leq\frac{1}{Df(a)}\Big|g_{\varepsilon}(a)\overset{b}{\underset{a}{\int}}
\left(\frac{b-y}{b-a}-\frac{1}{2}\right)dy\Big|+\\
&+&\frac{1}{Df(a)}\Big|\overset{b}{\underset{a}{\int}}|g_{\varepsilon}(y)-g_{\varepsilon}(a)|
\Big|\frac{b-y}{b-a}-\frac{1}{2}\Big|dy\Big|
+\frac{1}{Df(a)}\overset{b}{\underset{a}{\int}}
|\theta_{\varepsilon}(y)|\Big|\frac{b-y}{b-a}-\frac{1}{2}\Big|dy\leq\\
&\leq& M_1(b-a)\underset{x,t\in
[a,b]}{\max}|g_{\varepsilon}(x)-g_{\varepsilon}(t)|+M_1
\overset{b}{\underset{a}{\int}}|\theta_{\varepsilon}(y)|dy.
\end{eqnarray*}
Combining this  with the estimate (\ref{eq11})
we obtain the estimate (\ref{eq10}) for $G_1(a,b)$  in the Lemma..
\end{proof}
We can now finish the proof of Theorem \ref{theo3.1}. Combining
(\ref{eq9}) with the representations of $A(a,b)$ and $B(a,b)$ in Lemma \ref{lemm3.1} we obtain:\\

$Dist(z_1,z_2,z_3,z_4;f)=1+\overset{z_2}{\underset{z_1}{\int}}
\frac{D^{2}f(y)}{2Df(y)}dy+G_1(z_1,z_2)-\overset{z_3}{\underset{z_1}{\int}}
\frac{D^{2}f(y)}{2Df(y)}dy-G_1(z_1,z_3)$\\

$+\overset{z_4}{\underset{z_2}{\int}}
\frac{D^{2}f(y)}{2Df(y)}dy+G_2(z_3,z_4)
-\overset{z_4}{\underset{z_3}{\int}}
\frac{D^{2}f(y)}{2Df(y)}dy-G_2(z_3,z_4)
+O\left(\left(\overset{z_4}{\underset{z_1}{\int}}
|D^{2}f(y)|dy\right)^2\right)$\\

$=1+G_1(z_1,z_2)-G_1(z_1,z_3)+G_2(z_2,z_4)-G_2(z_3,z_4)+
+O\left(\left(\overset{z_4}{\underset{z_1}{\int}}
|D^{2}f(y)|dy\right)^2\right).$\\

Applying next (\ref{eq10}) for the intervals $[z_{s},z_{s+1}]\in [z_{1},z_{4}],$\,\,\,\,$ s=1,2,3 $ we obtain
$$
|G_1(z_{s},z_{s+1})|\leq \frac{M_1}{2}|z_4-z_1|\underset{[z_1,z_4]}{\max}
|g_{\varepsilon}(x)-g_{\varepsilon}(t)|
+\frac{M_1}{2}\overset{z_4}{\underset{z_1}{\int}}|\theta_{\varepsilon}(y)|dy+
\frac{M_1^2}{2}\overset{z_4}{\underset{z_1}{\int}}|f''(y)|dy.$$
from which the assertion of Theorem \ref{theo3.1} follows immedately.
\end{proof}

Next we consider the case when the interval $[z_1,z_4]$ contains just one
break point $x=x_b$. We estimate the distortion of the cross ratio when 
the break point lies outside the middle interval  $[z_2,z_3]$ i.e. $x_{b}\in [z_1,z_2]\cup[z_3,z_4] $.

For $z_i\in S^1$, $i=1,2,3,4$ with $z_1<z_2<z_3<z_4$ and $x_b\in [z_1,z_2]$ we set
\begin{equation}\label{eq12}
\alpha:=z_{2}-z_{1},\,\,   \beta:=z_{3}-z_{2},  \,\,\gamma:=z_{4}-z_{3},\,\, \tau:=z_{2}-x_{b},  \,\,\xi:=\frac{\beta}{\alpha},\,\,z:=\frac{\tau}{\alpha}.
\end{equation}

\begin{lemm}\label{lemm3.2} Assume, the function $f$ is defined 
on $[z_{1},z_{4})]$, its derivative
 $Df$ is continuous on every connected interval of 
the set  $[z_1,z_4]\backslash\{x_{b}\}$ and
$D^{2}f\in L^{1}[z_{1},z_{4})]$.  Choose $z_i\in S^1$, $i=1,2,3,4$, such that $z_1<z_2<z_3<z_4$
and $x_b\in[z_1,z_2]$. Then 
\begin{eqnarray}\label{eq13} &|Dist(z_1,z_2,z_3,z_4;f)-\frac{[\sigma(x_{b})+(1-\sigma(x_{b}))z](1+\xi)}{\sigma(x_{b})+(1-\sigma(x_{b}))z+\xi}|\leq
K_1 \overset{z_4}{\underset{z_1}{\int}}|D^{2}f(y)|dy,
\end{eqnarray}
where the constant $K_1>0$ depends only on the function $f$.
\end{lemm}
\begin{proof}. By assumption $x_b\in[z_1,z_2]$. Let the jump ratio of $Df(x)$
at the point $x_{b}$  be 
$\sigma(x_{b})=\frac{Df_{-}(x_{b})}{Df_{+}(x_{b})}$.
Rewrite then $Dist(z_1,z_2,z_3,z_4;f)$ in the form:
\begin{eqnarray*}
Dist(z_1,z_2,z_3,z_4;f)&=&\frac{Cr(f(z_1),f(z_2),f(z_3),f(z_4))}{Cr(z_1,z_2,z_3,z_4)}\\
&=&\Big(\frac{f(z_2)-f(z_1)}{z_2-z_1}:\frac{f(z_3)-f(z_1)}{z_3-z_1}\Big)\Big(\frac{f(z_4)-f(z_3)}{z_4-z_3}:\frac{f(z_4)-f(z_2)}{z_4-z_2}\Big).
\end{eqnarray*}
It is easy to check that
\begin{eqnarray*}
&&f(z_{2})-f(z_{1})=(f(x_{b})-f(z_{1}))+(f(z_{2})-f(x_{b}))=\\
&=&\left\lbrace Df_{-}(x_{b})(x_{b}-z_{1})-
\overset{x_{b}}{\underset{z_{1}}{\int}}
D^{2}f(y)(y-z_1)dy\right\rbrace + \left\lbrace Df_{+}(x_{b})(z_{2}-x_{b})+
\overset{z_{2}}{\underset{x_{b}}{\int}}
D^{2}f(y)(z_{2}-y)dy\right\rbrace =\\
&=&Df_{+}(x_{b})(z_{2}-z_{1})\left[ \sigma(x_{b})+(1-\sigma(x_{b}))\frac{\tau}{\alpha})\right]+\\
&+& Df_{+}(x_{b})\alpha\left \lbrace\frac{1}{Df_{+}(x_{b})}  \overset{z_{2}}{\underset{x_{b}}{\int}}
D^{2}f(y)\frac{z_{2}-y}{z_{2}-z_{1}}dy-\frac{1}{Df_{+}(x_{b})}  \overset{x_{b}}{\underset{z_{1}}{\int}}
D^{2}f(y)\frac{y-z_{1}}{z_{2}-z_{1}}dy\right\rbrace.\\
\end{eqnarray*}
Hence
\begin{equation}\label{eq14}
f(z_{2})-f(z_{1})=Df_{+}(x_{b})(z_{2}-z_{1}) \left[ \sigma(x_{b})+(1-\sigma(x_{b}))\frac{\tau}{\alpha}+r_{1}(x_{b},z_{1},z_{2})\right];
\end{equation}
in analogy we find
\begin{eqnarray*}
&&f(z_{3})-f(z_{1})=(f(x_{b})-f(z_{1}))+(f(z_{3})-f(x_{b}))=\\ 
&=&\left\lbrace Df_{-}(x_{b})(x_{b}-z_{1})-
\overset{x_{b}}{\underset{z_{1}}{\int}}
D^{2}f(y)(y-z_1)dy\right\rbrace + \left\lbrace Df_{+}(x_{b})(z_{3}-x_{b})+
\overset{z_{3}}{\underset{x_{b}}{\int}}
D^{2}f(y)(z_{3}-y)dy\right\rbrace=\\
&=&Df_{+}(x_{b}) (z_{3}-z_{1})\left[\frac{z_{3}-x_{b}}{z_{3}-z_{1}} +\sigma(x_{b})\frac{z_{1}-x_{b}}{z_{3}-z_{1}})\right]+\\ &+&Df_{+}(x_{b})(z_{3}-z_{1})\left \lbrace\frac{1}{Df_{+}(x_{b})}  \overset{z_{3}}{\underset{x_{b}}{\int}}
D^{2}f(y)\frac{z_{3}-y}{z_{2}-z_{1}}dy-\frac{1}{Df_{+}(x_{b})}  \overset{x_{b}}{\underset{z_{1}}{\int}}
D^{2}f(y)\frac{y-z_{1}}{z_{2}-z_{1}}dy\right\rbrace.
\end{eqnarray*}
respectively
\begin{equation}\label{eq15}
f(z_{3})-f(z_{1})=
Df_{+}(x_{b})(z_{3}-z_{1})
\left[\frac{\tau+\beta}{\alpha+\beta} +\sigma(x_b)\frac{\alpha-\tau}{\alpha+\beta}+r_{2}(x_{b},z_{1},z_{3})\right].
\end{equation}
For $|r_{1}(x_{b},z_{1},z_{2})|$ and $|r_{2}(x_{b},z_{1},z_{3})|$ then the following estimates hold:
\begin{equation}\label{eq16}
|r_{1}(x_{b},z_{1},z_{2})|,|r_{2}(x_{b},z_{1},z_{3})|\le \frac{1}{Df_{+}(x_{b})} \overset{z_4}{\underset{z_1}{\int}}|D^{2}f(y)|dy
\end{equation}
Using this, (\ref{eq14}) and (\ref{eq15}) we get:\\
\begin{eqnarray}\label{eq17}
 \Big|\frac{f(z_2)-f(z_1)}{z_2-z_1}:\frac{f(z_3)-f(z_1)}{z_3-z_1}-
\frac{[\sigma(f)+(1-\sigma(f))z](1+\xi)}{\sigma(f)+\xi+(1-\sigma(f))z}\Big|\nonumber\\
\leq
K_{2}\overset{z_3}{\underset{z_1}{\int}} |D^{2}f(y)|dy,
\end{eqnarray}
with $\xi$ and  $z$ as defined in (\ref{eq12}) and where the constant $K_{2}>0$ is depending only on the function $f$. 

Since the interval  $[z_2,z_4]$ does not contain the break point
$x_{b}$,  it can easily be shown that
\begin{eqnarray*}
\Big|\frac{f(z_4)-f(z_3)}{z_4-z_3}:\frac{f(z_4)-f(z_2)}{z_4-z_2}-1\Big|\leq
K_{3}\overset{z_4}{\underset{z_2}{\int}} |D^{2}f(y)|dy,
\end{eqnarray*}
where also  the constant $K_{3}>0$ depends only on $f$.
The last inequality together with the bounds (\ref{eq16}) and (\ref{eq17})
imply the assertion of Lemma \ref{lemm3.2}. 
\end{proof}
\begin{rem} If the break point $x=x_{b}$ belongs to the right interval $[z_{3},z_{4}]$,
then one can prove the following estimate: 
$$ |Dist(z_{1},z_{2},z_{3},z_{4};f)-\frac{[\sigma(x_b)+(1-\sigma(x_b))\vartheta](1+\eta)}{\sigma(x_b)+(1-\sigma(x_b))\vartheta+\eta}|\leq
K_4 \overset{z_4}{\underset{z_1}{\int}}|D^{2}f(y)|dy,
$$
where $\eta=\frac{z_{3}-z_{2}}{z_{4}-z_{3}}$, $\vartheta=\frac{x_{b}-z_{3}}{z_{4}-z_{3}}$ and the constant $K_4>0$ depends only on the function $f$.
\end{rem}

\section{The proofs of Theorem \ref{DLM} and Theorem \ref{DLM1}}
For the proofs of Theorem \ref{DLM} and Theorem \ref{DLM1} we need several Lemmas which we formulate next and whose proofs will be given later.
\begin{lemm}\label{lemm4.1} Assume that the lift $\varphi$ of the 
conjugating homeomorphism  $T_{\varphi}(x)$ has a positive derivative
$D \varphi(x_{0})=\omega$ at the point $x=x_{0}\in S^{1}$, and the
following conditions hold for $z_i\in S^1,\, i=1,..,4$ with $z_1<z_2<z_3<z_4$ and some constant $R_{1}>1:$
\begin{itemize}
\item[(a)] $R_{1}^{-1}|z_{3}-z_{2}| \leq|z_{2}-z_{1}|\leq
R_{1}|z_{3}-z_{2}|$,\, 
$R_{1}^{-1}|z_{3}-z_{2}| \leq|z_{4}-z_{3}|\leq
R_{1}|z_{3}-z_{2}|;$
\item[(b)] $\underset{1\leq i\leq 4}{max}|z_{i}-x_{0}|\leq
R_{1}|z_{2}-z_{1}|.$
\end{itemize}
Then for any $\varepsilon>0$ there exists
$\delta=\delta(\varepsilon)>0$ such  that
\begin{eqnarray}\label{eq18}
|Dist(z_1,z_2,z_3,z_4;T_{\varphi})-1|\leq C_{2}\varepsilon,
\end{eqnarray}
if $z_{i}\in (x_{0}-\delta,\ x_{0}+\delta)$ for all $i=1,2,3,4$, 
 where the constant $C_{2}>0$ depends only on $R_{1},$ 
$\omega$ and not on $\varepsilon.$
\end{lemm}
Suppose that $D\varphi(x_{0})=\omega$ for some point $x=x_{0},\ x_{0}\in S^{1}$.
Consider its $n-$th dynamical partition
$$\xi_{n}(x_{0})=\left\lbrace \Delta_{i}^{(n-1)}(x_{0}), \  0\leq
i<q_{n}; \    \Delta_{j}^{(n)}(x_{0}), \ 0\leq j<q_{n-1}\right\rbrace. $$
For definitness suppose, that $n$ is odd. Then
$\Delta_{0}^{(n)}(x_{0})=[T_{f}^{q_{n}}x_{0},x_{0}]$ and $\Delta_{0}^{(n-1)}(x_{0})=[x_{0},T_{f}^{q_{n-1}}x_{0}]$ are its two generators.
Denote by $\overline{a}_{b}$ and $\overline{c}_{b}$ the preimages of $a_{b}$ and $c_{b}$ in the
interval $\left[ T_{f}^{q_{n}}x_{0}, T_{f}^{q_{n-1}}x_{0}\right] $ 
 such that $\overline{a}_{b}=T_{f}^{-l}a_{b}$ and $\overline{c}_{b}=T_{f}^{-p}c_{b}$
for some $l,p \in [0,q_{n}) $.

Define next for $m \in [0,q_{n}]$
\begin{eqnarray}\label{eq19}
\xi(m):=\frac{T_{f}^{m}z_{3}-T_{f}^{m}z_{2}}{T_{f}^{m}z_{2}-T_{f}^{m}z_{1}},\,\,\,  \,\,\,
z(m):=\frac{T_{f}^{m}z_{2}-T_{f}^{m}\overline{c}_{b}}{T_{f}^{m}z_{2}-T_{f}^{m}z_{1}},\nonumber\\
\eta(m):=\frac{T_{f}^{m}z_{3}-T_{f}^{m}z_{2}}{T_{f}^{m}z_{4}-T_{f}^{m}z_{3}},\,\,\,\,\,\,  
\vartheta(m):=\frac{T_{f}^{m}\overline{c}_{b}-T_{f}^{m}z_{3}}{T_{f}^{m}z_{4}-T_{f}^{m}z_{3}}
\end{eqnarray}

The numbers $z(m)$ (if  $\overline{c}_{b}\in \left[z_{1}, z_{2} \right]$) and $\vartheta(m)$
(if $\overline{c}_{b}\in \left[z_{3}, z_{4} \right] $) are called normalized coordinates of the point $T_{f}\overline{c}_{b}$.
It is clear that the normalized coordinates $z(m)$ (respectively $\vartheta(m)$) change from 0 to 1, when the break point $c_{b}$ is moving from $T_{f}^{p}z_{2}$ to $T_{f}^{p}z_{1}$ 
(respectively from $T_{f}^{p}z_{3}$ to $T_{f}^{p}z_{4}$).

\begin{defi}\label{defi4.1} The intervals  $ \left\lbrace T_{f}^{j}[z_{1},z_{2}],  T_{f}^{j}[z_{2},z_{3}], T_{f}^{j}[z_{3},z_{4}]: 0\le j\le q_{n}-1\right\rbrace $ cover the 
break points  $a_{b}$, $c_{b}$ regularly with constants $C\geq 1$, $\zeta\in [0,1]$, if
\begin{itemize}
\item[1)] the intervals  $ \left\lbrace T_{f}^{j}[z_{1},z_{4}], 0\le j\le q_{n}-1\right\rbrace$ cover every break point only once;

\item[2)] either  $z_{2}=T_{f}^{-l}a_{b}$  and  $T_{f}^{-p}c_{b}\in \left[z_{1}, z_{2} \right] $ 
or  $z_{3}=T_{f}^{-l}a_{b}$ 
and  $T_{f}^{-p}c_{b}\in \left[z_{3}, z_{4} \right] $ for some $l, p\in \left[ 0, q_{n}\right) $;

\item[3)] $\xi(0)\geq C$  and \,\,$z(0)\in [0, \zeta]$ \,\, if\,\, 
$\overline{c}_{b}=T_{f}^{-l}c_{b}\in \left[z_{1}, z_{2} \right], $\\
$\eta(0) \geq C$ and $\vartheta(0)\in [0, \zeta]$ \,\, if\,\, $\overline{c}_{b}=T_{f}^{-p}c_{b}\in \left[z_{3}, z_{4} \right], $
\end{itemize}
or if conditions 1)-3) hold for $a_{b}$ and $c_{b}$ interchanged.\\
\end{defi}
In order to formulate the next Lemma we introduce the following functions for $x>0$ and $0\leq t\leq 1$:
\begin{eqnarray}\label{eq20}
G(x)=\frac{\sigma(a_{b})(1+x)}{\sigma (a_{b})+x},\,\,
F(x,t)=\frac{[\sigma(c_{b})+(1-\sigma(c_{b}))t](1+x)}{\sigma(c_{b})+(1-\sigma(c_{b}))t+x},
\end{eqnarray}
\begin{lemm}\label{lemm4.2} Suppose that the homeomorphism $T_{f}$ satisfies the conditions of Theorem \ref{DLM}. If
the intervals  $ \left\lbrace T_{f}^{j}[z_{1},z_{2}],  T_{f}^{j}[z_{2},z_{3}], T_{f}^{j}[z_{3},z_{4}]: 0\le j\le q_{n}-1\right\rbrace $ cover 
the break points  $a_{b}$, $c_{b}$ such that either $z_2=T_f^{-l}a_b$ and $T_f^{-p}c_b\in[z_1,z_2]$ or $z_3=T_f^{-l}a_b$ and  $T_f^{-p}c_b\in[z_3,z_4]$ then 
\begin{itemize}
\item[I)]
$ Dist(z_{1},z_{2},z_{3},z_{4};T_{f}^{q_{n}})
=[G(\xi(l))+\chi_{1}][F(\xi(p),z(p))+\chi_{2}]\times $\\
$\times\prod_{\substack{0\leq i<q_{n} \\ i\neq l,p }}Dist (T_{f}^{i}z_{1},T_{f}^{i}z_{2}, T_{f}^{i}z_{3},T_{f}^{i}z_{4};T_{f})
$ if  $z_{2}=\overline{a}_{b},$ $ \overline{c}_{b}\in \left[z_{1}, z_{2} \right];  $ 
\item[II)]
$ Dist(z_{1},z_{2},z_{3},z_{4};T_{f}^{q_{n}})
=[G(\eta(l))+\chi_{3}][F(\eta(p),\vartheta(p))+\chi_{4}]\times$ \\
$\times\prod_{\substack{0\leq i<q_{n} \\ i\neq l,p }}Dist (T_{f}^{i}z_{1},T_{f}^{i}z_{2}, T_{f}^{i}z_{3},T_{f}^{i}z_{4};T_{f})
$ if  $z_{3}=\overline{a}_{b},$ $ \overline{c}_{b}\in \left[z_{3}, z_{4} \right]  $
\end{itemize}
where 
\begin{equation}\label{eq21}
|\chi_{j}|=|\chi_{j}(z_{1},z_{2},z_{3},z_{4})|\leq
K_1 \overset{z_4}{\underset{z_1}{\int}}|D^{2}f(y)|dy,\,\, 1\leq j \leq4.
\end{equation}
and the constant $R_{5}$ does not depend on $n$ and $\varepsilon$
\end{lemm}
Using Lemma \ref{lemm2.2} we get the following inequalities for all $1\leq m\leq q_n$:
\begin{eqnarray}\label{eq22}
e^{-v}\xi(0)\leq \xi(m)\leq e^{v} \xi(0),\,\,\,\,
e^{-v}z(0)\leq z(m)\leq e^{v}z(0),\nonumber\\
e^{-v}\eta(0)\leq \eta(m)\leq e^{v} \eta(0),\,\,\,\, 
e^{-v}\vartheta(0)\leq \vartheta(m)\leq e^{v}\vartheta(0)
\end{eqnarray}
From this it follows, that the normalized coordinates  $\xi(m)$,  $\eta(m)$, $z(m)$, and  $\vartheta(m)$ are uniformly (with respect to $x_{0}$ and $m$) comparable with the initial normalized coordinates $\xi(0)$,  $\eta(0)$, $z(0)$, and  $\vartheta(0)$ respectively.

\begin{lemm}\label{lemm4.3}If a circle homeomorphism $T_{f}$ satisfies the conditions of Lemma \ref{lemm2.2} then there exist for any $x_{0}\in S^{1}$ and any $\delta>0$ constants  $C_{0}=C_{0}(f,\sigma(a_{b}),\sigma(c_{b}))>1$ and $\zeta_{0}=\zeta_{0}(f,\sigma(a_{b}),\sigma(c_{b}))\in (0,1)$, such that for all triple of intervals $[z_{s},z_{s+1}]\subset(x_{0}-\delta, \ x_{0}+\delta),$ $s=1,2,3,$ covering the break points $a_{b}$,\,\,$c_{b}$ regularly with constants $C_{0}$ 
and $\zeta_{0}$ the following relations hold:
\begin{eqnarray*}
|G(\xi(l))F(\xi(p),z(p))-1|\geq \frac{|\sigma(a_{b})\sigma(c_{b})-1|}{4} \,\,
if \,\, z_{2}=\overline{a}_{b},\,0\le \frac{z_{2}-\overline{c}_{b}}{z_{2}-z_{1}}\le \zeta_{0}
\end{eqnarray*}
 respectively
\begin{eqnarray*}
|G(\eta(l))F(\eta(p),\vartheta(p))-1|\geq \frac{|\sigma(a_{b})\sigma(c_{b})-1|}{4}
\,\, if \,\, z_{3}=\overline{a}_{b}, \, 0\le \frac{\overline{c}_{b}-z_{3}}{z_{4}-z_{3}}\le \zeta_{0}.
\end{eqnarray*}
\end{lemm}
\begin{lemm}\label{lemm4.4} If the circle homeomorphism $T_{f}$ satisfies the conditions of Lemma \ref{lemm2.2} there exists for any $x_{0}\in S^{1}$ and any $\delta>0$ a number $N=N(\delta, x_{0})>1,$ such that for all $n>N(\delta, x_{0})$,
there is a triple of intervals $[z_{s},z_{s+1}]\subset(x_{0}-\delta, \ x_{0}+\delta),$ $s=1,2,3$ with the following properties:
\begin{itemize}
\item[1)] the interval $[z_{1},z_{4}]$ is $q_{n}$  small;
\item[2)] the intervals $[z_{s},z_{s+1}]$  and $ [T_{f}^{q_{n}}z_{s}, T_{f}^{q_{n}}z_{s+1}]$\, 
$s=1,2,3$ satisfy conditions  a) and b) of Lemma 4.1 with some  constant $R_{1}>1$ 
depending on  $C_{0}$, $\zeta_{0}$ and $v$; 
\item[3)] the intervals $\left\lbrace T_{f}^{i}[z_{1},z_{2}],\,\,T_{f}^{i}[z_{2},z_{3}]\,\, T_{f}^{i}[z_{3},z_{4}],\,\,\,0\leq i \leq q_{n}-1 \right\rbrace $\, either cover both break 
points $a_{b}$, $c_{b}$ regularly with constants $C_{0}$ and
  $\zeta_{0}$, or cover only the break point $a_{b}$ such that
 its preimage $\overline{a}_{b}$ coincides with $z_{2}$ or $z_{3};$
\end{itemize}
\end{lemm}
\begin{lemm}\label{lemm4.5} Suppose, that the circle homeomorphism $T_{f}$ satisfies the conditions of Theorem \ref{DLM} and the intervals $[z_{s},z_{s+1}],\,\, s=1,2,3$ satisfy conditions 1)-3) of Lemma \ref{lemm4.4}.
Then the following inequality holds for sufficiently large $n$:
 $$|Dist(z_{1},z_{2},z_{3},z_{4};T_{f}^{q_{n}})-1|>const >0 $$
where the constant depends only on the function $f$.
\end{lemm}
After these preparations we  can now proceed to the proof of  Theorem \ref{DLM}.
Let $T_{f}$ be a class $P$-homeomorphism satisfying the conditions of Theorem \ref{DLM}.
Since its rotation number $\rho_{f}$ is irrational the $T_{f}$-invariant measure $\mu_{f}$
is nonatomic  i.e. every one point subset of the circle has zero  $\mu_{f}$-measure.
 The conjugating map $T_{\varphi}$ related to 
$\mu_{f}$ by $T_{\varphi}x=\mu_{f}([0,x])$ , $x\in S^{1}$, is a continuous and monotone 
increasing function on $S^{1}$. Hence $T_{\varphi}$ has a finite derivative 
almost everywhere (w.r.t. Lebesgue measure) on the circle. We show that
$D\varphi(x)=0$ at all points at which the derivative is defined. 
Choose an $\varepsilon>0$ and 
a triple of intervals $[z_{s},z_{s+1}]\subset(x_{0}-\delta, \ x_{0}+\delta),$\,\, $s=1,2,3,$
satisfying the conditions of Lemma \ref{lemm4.1}. It follows from this Lemma  and Lemma \ref{lemm4.4} that
\begin{eqnarray}\label{eq23}
|Dist(z_1,z_2,z_3,z_4;T_{\varphi})-1|\leq C_{2}\varepsilon,
\end{eqnarray}
and
\begin{eqnarray}\label{eq24}
|Dist(T_{f}^{q_{n}}z_{1},T_{f}^{q_{n}}z_{2}, T_{f}^{q_{n}}z_{3},T_{f}^{q_{n}}z_{4}; T_{\varphi})-1|\leq C_{2}\varepsilon.
\end{eqnarray}
By definition 
\begin{eqnarray}\label{eq25}
&&Dist(T_{f}^{q_{n}}z_{1},T_{f}^{q_{n}}z_{2}, T_{f}^{q_{n}}z_{3},T_{f}^{q_{n}}z_{4}; T_{\varphi})=\nonumber\\
&=&\frac{Cr(T_{\varphi}(T_{f}^{q_{n}}z_{1}), T_{\varphi}(T_{f}^{q_{n}}z_{2}),   T_{\varphi}(T_{f}^{q_{n}}z_{3}), T_{\varphi}(T_{f}^{q_{n}}z_{4}) )}{Cr(T_{f}^{q_{n}}z_{1},T_{f}^{q_{n}}z_{2}, T_{f}^{q_{n}}z_{3},T_{f}^{q_{n}}z_{4})}.
\end{eqnarray}\\
Since $T_{\varphi}$   conjugates $T_{f}$ with the linear rotation  $T_{\rho}$, we can readily see that
$$Cr(T_{\varphi}(T_{f}^{q_{n}}z_{1}),(T_{\varphi}(T_{f}^{q_{n}}z_{2}), (T_{\varphi}(T_{f}^{q_{n}}z_{3}),(T_{\varphi}(T_{f}^{q_{n}}z_{4}) ))=Cr(T_{\varphi}z_{1},T_{\varphi}z_{2}, T_{\varphi}z_{3},T_{\varphi}z_4) ))$$
and hence
$$
Dist(T_{f}^{q_{n}}z_{1},T_{f}^{q_{n}}z_{2}, T_{f}^{q_{n}}z_{3},T_{f}^{q_{n}}z_{4}; T_{\varphi})
=\frac{Cr(T_{\varphi}z_{1},T_{\varphi}z_{2}, T_{\varphi}z_{3},T_{\varphi}z_4) ))}{Cr(T_{f}^{q_{n}}z_{1},T_{f}^{q_{n}}z_{2}, T_{f}^{q_{n}}z_{3},T_{f}^{q_{n}}z_{4})}$$
This together with  \eqref{eq23}, \eqref{eq24} and \eqref{eq25} implies 
\begin{eqnarray}\label{eq26}
|Dist(z_1,z_2,z_3,z_4;T_{f}^{q_{n}})-1|\leq C_{3}\varepsilon,
\end{eqnarray}\\
where the constant $C_{3}>0$ does not depend on $\varepsilon$ and $n$. But this contradicts Lemma \ref{lemm4.5} according to which
$$| Dist(z_1,z_2,z_3,z_4;T_f^{q_n}-1|>const >0 $$
for sufficiently large $n$.
 This contradiction proves Theorem \ref{DLM}. 

\section{The proofs of Lemmas \ref{lemm4.1}-\ref{lemm4.5}}

\textbf{Proof of Lemma \ref{lemm4.1}}  Suppose, that the derivative
$D\varphi(x_{0})$ exists and $D\varphi(x_{0})=\omega >0$. By the
definition of the derivative there exists for any  $\varepsilon>0$  a number \,\, 
$\delta=\delta(x_{0} a \ \varepsilon)>0$, such that for all $x\in(x_{0}-\delta, \ x_{0}+\delta).$
\begin{eqnarray}\label{eq27}
\omega-\varepsilon<\frac{\varphi(x)-\varphi(x_{0})}{x-x_{0}}<\omega+\varepsilon.\end{eqnarray}
Now take four points $z_{i}\in(x_{0}-\delta,x_{0}+\delta)$ 
satisfying conditions (a) and (b) of Lemma \ref{lemm4.1}. 
Assume that $z_{i}<x_{0}, \,1\leq i\leq 4$. For the other cases Lemma \ref{lemm4.1} can be proved similarly. Relation
(\ref{eq27}) implies for $x=z_{i},$ $ i=1,2,3,4$
$$
(\omega-\varepsilon)(x_{0}-z_{i})<\varphi(x_{0})-\varphi(z_{i})<
(\omega +\varepsilon)(x_{0}-z_{i}).
$$

This  yields the following inequalities:

\begin{eqnarray}\label{eq28}
\omega-\varepsilon\frac{(x_{0}-z_{s+1})+(x_{0}-z_{s})}{z_{s+1}-z_{s}}&\leq&
\frac{\varphi(z_{s+1})-\varphi(z_{s})}{z_{s+1}-z_{s}}\nonumber\\
&\leq&\omega+\varepsilon
\frac{(x_{0}-z_{s+1})+(x_{0}-z_{s})}{z_{s+1}-z_{s}}
\end{eqnarray}
for $s=1,2,3$, and
\begin{eqnarray}\label{eq29}
\omega -\varepsilon\frac{(x_{0}-z_{s+2})+(x_{0}-z_{s})}{z_{s+2}-z_{s}}&\leq&
\frac{\varphi(z_{s+2})-\varphi(z_{s})}{z_{s+2}-z_{s}}\nonumber\\
&\leq&\omega+\varepsilon\frac{(x_{0}-z_{s+2})+(x_{0}-z_{s})}{z_{s+2}-z_{s}}
\end{eqnarray}
for $s=1,2.$\\

From conditions (a) and (b) of Lemma \ref{lemm4.1} on the other hand  it follows that
\begin{eqnarray}\label{eq30}
\underset{1\leq i\leq 4}{max}\Big\{\frac{x-z_{i}}{z_{2}-z_{1}},
 \frac{x_{0}-z_{i}}{z_{3}-z_{1}}, \frac{x_{0}-z_{i}}{z_{4}-z_{2}}, \frac{x_{0}-z_{i}}{z_{4}-z_{3}}
\Big\}\leq K_{1}
\end{eqnarray}
where the constant $K_{1}>0$ depends on $R_{1}$ and 
does not depend on $\varepsilon$.\\
We rewrite $Dist(z_1,z_2,z_3,z_4;T_{\varphi})$ in the following form:
$$
Dist(z_1,z_2,z_3,z_4;T_{\varphi})
=\frac{T_{\varphi}z_{2}-T_{\varphi}z_{1}}{z_{2}-z_{1}}\cdot
\frac{T_{\varphi}z_{4}-T_{\varphi}z_{3}}{z_{4}-z_{3}}\cdot
\frac{z_{3}-z_{1}}{T_{\varphi}z_{3}-T_{\varphi}z_{1}}\cdot
\frac{z_{4}-z_{2}}{T_{\varphi}z_{4}-T_{\varphi}z_{2}}.
$$
The inequalities (\ref{eq28})-(\ref{eq30}) then  imply the assertion of Lemma \ref{lemm4.1}.\\

\textbf{Proof of Lemma \ref{lemm4.2}}.  We consider the case $z_{2}=T_{f}^{-l}a_{b}$,\,\, $ T_{f}^{-p}c_{b}\in \left[z_{1}, z_{2} \right]$, \,\,  $0 \leq l,p\leq q_n$, the case  $z_{3}=T_{f}^{-l}a_{b},$\,\, $ T_{f}^{-p}c_{b}\in \left[z_{3}, z_{4} \right],$\,\,  $0 \leq l,p\leq q_n$, can be treated similarly.
Rewrite the distortion $Dist(z_{1},z_{2},z_{3},z_{4};T_{f}^{q_{n}})$ in the following form
\begin{eqnarray}\label{eq31}
&&Dist(z_{1},z_{2},z_{3},z_{4};T_{f}^{q_{n}})
=Dist (T_{f}^{l}z_{1},T_{f}^{l}z_{2}, T_{f}^{l}z_{3},T_{f}^{l}z_{4};T_{f})\times \notag\\
&\times& Dist(T_{f}^{p}z_{1},T_{f}^{p}z_{2}, T_{f}^{p}z_{3},T_{f}^{p}z_{4};T_{f})
\prod_{\substack{0\leq i<q_{n} \\ i\neq l,p }}Dist (T_{f}^{i}z_{1},T_{f}^{i}z_{2}, T_{f}^{i}z_{3},T_{f}^{i}z_{4};T_{f}).
\end{eqnarray}
By assumption only the two intervals $[T_{f}^{l}z_{1},T_{f}^{l}z_{2}]$ and  $[T_{f}^{p}z_{1},T_{f}^{p}z_{2}]$
contain the break points: namely $a_{b}=T_{f}^{l}z_{2}=a_{b}$, and $ c_{b}\in [T_{f}^{p}z_{1}, T_{f}^{p}z_{2} ]$ for some  $l,p\in [0,q_{n}).$

Using Lemma \ref{lemm3.2} and the definitions of the functions  $G(x)$,\,\,$F(x,t)$ we get 
\begin{eqnarray*}
Dist (T_{f}^{l}z_{1},T_{f}^{l}z_{2}, T_{f}^{l}z_{3},T_{f}^{l}z_{4};T_{f})&=&\frac{\sigma(a_{b})(1+\xi(l))}{\sigma (a_{b})+\xi(l)}+\chi_{1}=G(\xi(l)+\chi_{1},\\
Dist (T_{f}^{p}z_{1},T_{f}^{p}z_{2}, T_{f}^{p}z_{3},T_{f}^{p}z_{4};T_{f})&=&\frac{[\sigma(c_{b})+(1-\sigma(c_{b}))z(p)](1+\xi(p))}{\sigma(c_{b})+(1-\sigma(c_{b}))z(p)+\xi(p)}+\chi_{2}=\\
&=&F(\xi(p),z(p))+\chi_{2},
\end{eqnarray*}
with $|\chi_{j}|=|\chi_{j}(z_{1},z_{2},z_{3},z_{4})|\leq
K_1 \overset{z_4}{\underset{z_1}{\int}}|D^{2}f(y)|dy,$\,\,\,  $j=1,2.$\\
This together with (\ref{eq31}) imply the assertion of Lemma \ref{lemm4.2}.

\textbf{Proof of Lemma \ref{lemm4.3}}.  We prove only the bound for
$G(\xi(l)F(\xi(p),z(p))$. The one for $G(\eta(l)F(\eta(p),\vartheta(p))$ can be proved similarly. We start rewriting
$G(\xi(l)F(\xi(p),z(p))$ in the following form:
\begin{eqnarray}\label{eq32}
G(\xi(l)F(\xi(p),z(p))=\frac{\sigma(a_{b})(1+\xi(l))}{\sigma (a_{b})+\xi(l)}\cdot
\frac{[\sigma(c_{b})+(1-\sigma(c_{b}))z(p)](1+\xi(p)]}{\sigma(c_{b})+(1-\sigma(c_{b}))z(p)+\xi(p)}\nonumber =\\ 
=[\sigma(a_{b})\sigma(c_{b})+(1-\sigma(c_{b}))\sigma(a_{b})z(p)]
\times \big [\frac{(1+\xi(l))}{\sigma (a_{b})+\xi(l)}\nonumber\\ \cdot\frac{(1+\xi(p))}{\sigma(c_{b})+(1-\sigma(c_{b}))z(p)+\xi(p)}\big ] \equiv \Phi_{1}(z(p))\times \Phi_{2}(\xi(l),\xi(p),z(p)) 
\end{eqnarray}
where $z(p)\in [0,1])$ and  $\xi(l),\, \xi(p)>0.$

It is clear, that 
$\Phi_{1}(0)=\sigma(a_{b})\sigma(c_{b})$ and  
$ \Phi_{2}(\xi(l),\xi(p),z(p))$ tends to 1 as  $\xi(l), \xi(p)$ tend to $\infty.$
Recall that $\sigma{a_{b}}\sigma{c_{b}}\neq 1$ by assumption.\\
Next we discuss the conditions under which  the expression  $\Phi_{1}(z(p))\Phi_{2}(\xi(l),\xi(p),z(p))$ stays
away from 1.
Obviously
\begin{eqnarray}\label{eq57}
|\Phi_{1}\Phi_{2}-1|=|(\Phi_{1}-1)+\Phi_{1}(\Phi_{2}-1)|\geq||\Phi_{1}-1|-\Phi_{1}|\Phi_{2}-1|.
\end{eqnarray}
Using the bounds for $z(m)$ in (\ref{eq22}) we get
\begin{eqnarray*}
|\Phi_{1}-1|&=&|\sigma(a_{b})\sigma(c_{b})+(1-\sigma(c_{b}))\sigma(a_{b})z(p)-1|\\ 
&\geq& |\sigma(a_{b})\sigma(c_{b})-1|-|1-\sigma(c_{b})|\sigma(a_{b})z(p)\\
&\geq& |\sigma(a_{b})\sigma(c_{b})-1|-
|1-\sigma(c_{b})|\sigma(a_{b})e^{v}z(0).
\end{eqnarray*}
If next $z(0)$  fulfills the inequality
$$|\sigma(a_{b})\sigma(c_{b})-1|-
|1-\sigma(c_{b})|\sigma(a_{b})e^{v}z(0)\geq
 \frac{|\sigma(a_{b})\sigma(c_{b})-1|}{2},$$
and hence
$$z(0)\leq \frac{|\sigma(a_{b})\sigma(c_{b})-1|}{2e^{v}|\sigma(a_{b})\sigma(c_{b})-\sigma(a_{b})|},$$
 then we conclude that 
\begin{eqnarray}\label{eq34}
|\Phi_{1}-1| \geq \frac{|\sigma(a_{b})\sigma(c_{b})-1|}{2},\,\,\,\, if\,\,\,0\le z(0)\leq
\zeta_{0},
\end{eqnarray}
where
\begin{equation}\label{35}
\zeta_{0}:=min \left\lbrace \frac{|\sigma(a_{b})\sigma(c_{b})-1|}{2e^{v}|\sigma(a_{b})\sigma(c_{b})-\sigma(a_{b})|}, 1\right\rbrace .
\end{equation}

Next we determine, under which condition on  $\xi(0)$ the inequality:
\begin{eqnarray}\label{eq36}
\Phi_{1}|\Phi_{2}-1|\le \frac{|\sigma(a_{b})\sigma(c_{b})-1|}{4}
\end{eqnarray}
holds true. Obviously, for  $z(p)\in [0,1].$ one has 
$\Phi_{1}(z(p))\leq max\left\lbrace \sigma(a_{b})\sigma(c_{b}),\sigma(c_{b})\right\rbrace:=  m_{\sigma} ,$
 for  $z(p)\in [0,1].$
Inequality (\ref{eq36}) then follows, if 
\begin{eqnarray}\label{eq37}
|\Phi_{2}-1| \le \frac{|\sigma(a_{b})\sigma(c_{b})-1|}{4m_{\sigma}}.
\end{eqnarray}
Now, if $\xi(l)$ and $\xi(p)$ are sufficiently large, then, since
\begin{eqnarray}\label{eq38}
\Phi_{2}-1=\frac{(1+\xi(l))}{\sigma (a_{b})+\xi(l)}\cdot\frac{(1+\xi(p))}{\sigma(c_{b})+(1-\sigma(c_{b}))z(p)+\xi(p)}-1,
\end{eqnarray}
 the right hand side of (\ref{eq38}) behaves like 
$(1+O(\frac{1}{\xi(l)}))(1+O(\frac{1}{\xi(p)}))-1$, 
which can be bounded by 
$R_{6}\left( \frac{1}{\xi(l)}+\frac{1}{\xi(p)}\right) $
for some  constant $R_{6}>1$  not depending on  $\xi(l)$ and  $\xi(p).$
On the other hand, according to relations (\ref{eq22}),  $\xi(m)$ is for $ m\in (0,q_{n}]$  comparable with $\xi(0)$, and hence
\begin{eqnarray}\label{eq39}
|\Phi_{2}-1|\leq R_{6}\left( \frac{1}{\xi(l)}+\frac{1}{\xi(p)}\right) \leq 2R_{6}e^{v}\frac{1}{\xi(0)}.
\end{eqnarray}
Hence if 
$$2R_{6}e^{v}\frac{1}{\xi(0)}\leq\frac{|\sigma(a_{b})\sigma(c_{b})-1|}{4m_{\sigma}}$$
respectively
\begin{eqnarray}\label{eq40}
 \xi(0)\geq \frac{4R_{6}e^{v}m_{\sigma}}{|\sigma(a_{b})\sigma(c_{b})-1|},
\end{eqnarray}
then  inequality (\ref{eq37}) holds true. Finally, we define the constant $C_{0}$ by
\begin{equation}\label{eq41}
C_{0}:=\max \left\lbrace \frac{4R_{6}e^{v}m_{\sigma}}{|\sigma(a_{b})\sigma(c_{b})-1|},1\right\rbrace . 
\end{equation}
From (\ref{eq34})-(\ref{eq41})  the assertion of Lemma \ref{lemm4.3} then follows  immediately.\\

{\textbf{Proof of Lemma \ref{lemm4.4}}}. Let $D\varphi(x_{0})=\omega >0.$ Fix $n\geq1$. W.l.o.g. we consider the case  $n $ odd, the case $n$ even can be deduced from the odd case by reversing the orientation of the circle. From the structure of the dynamical partition $\xi_n(x_{0})$ it follows, that both preimages $\overline{a}_{b}$ , $\overline{c}_{b}$ are in the interval  $[T_{f}^{q_{n}}x_{0},T_{f}^{q_{n-1}}x_{0}]$ with
$\overline {a}_{b}=T_{f}^{-l}{a}_{b}$, $\overline{c}_{b}=T_{f}^{-p}{c}_{b}$ for some $0\leq l,p\leq q_{n}$.
Take the point $\overline{a}_{b}$ and consider its neighbourhood $[T_{f}^{-q_{n-1}}\overline{a}_{b},T_{f}^{q_{n-1}}\overline{a}_{b}] $. By Corollary \ref{cor2.2}
for any   $a,b\in S^{1}$  all the intervals $[a,b]$, $T_{f}^{q_{n}}[a,b]$, $T_{f}^{-q_{n}}[a,b]$ are $e^{v}$- comparable. Since
 $\overline{a}_{b}\in [T_{f}^{q_{n}}x_{0},T_{f}^{q_{n-1}}x_{0}] $,
it can easily be shown, that the  pairs of intervals 
$([T_{f}^{-q_{n-1}}x_{0},x_{0}], [T_{f}^{-q_{n-1}}\overline{a}_{b},\overline{a}_{b}])$,
$([x_{0},T_{f}^{q_{n-1}}x_{0}], [\overline{a}_{b},T_{f}^{q_{n-1}}\overline{a}_{b}]) $ and 
$([T_{f}^{-q_{n-1}}x_{0},T_{f}^{q_{n-1}}x_{0}], [T_{f}^{-q_{n-1}}\overline{a}_{b},T_{f}^{q_{n-1}}\overline{a}_{b}]) $
are $e^{v}$- comparable.

Let $\tau_{0}$ be the middle point of the interval $[T{f}^{-q_{n-1}}\overline{a}_{b},\overline{a}_{b} ].$
Since
$[\overline{a}_{b},T_{f}^{q_{n-1}}\tau_{0}]=\\ T_{f}^{q_{n-1}}[T_{f}^{-q_{n-1}}\overline{a}_{b},\tau_{0}]$ and  $l([T_{f}^{-q_{n-1}}\overline{a}_{b},\tau_{0}])=l([\tau_{0},\overline{a}_{b}])$
we conclude, that the intervals $[\tau_{0},\overline{a}_{b}]$ and  $[\overline{a}_{b},T_{f}^{q_{n-1}}\tau_{0}]$
are $e^{v}$- comparable (see figure $1$).
\begin{figure}
\centering
\includegraphics[width=14cm]{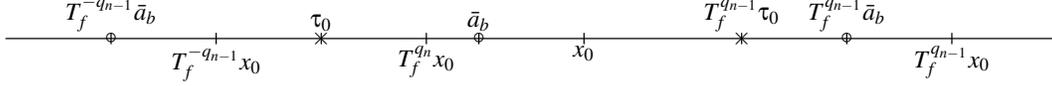}
\caption{Comparison of the intervals $[x_{0},T_{f}^{q_{n-1}} x_0 ]$
  and  $[\bar{a}_{b},T_{f}^{q_{n-1}}\bar{a}_{b} ]$}
\end{figure}

\begin{figure}
\centering
\includegraphics[width=14cm]{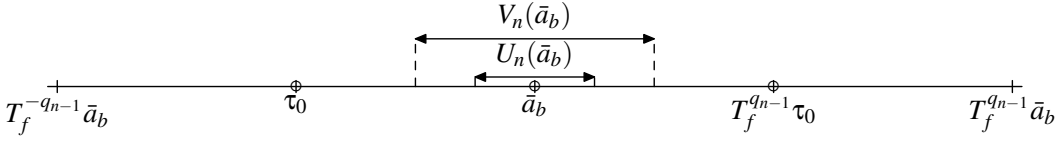}
\caption{The neighbourhoods $U_{n}(\bar{a}_{b})$ and $V_{n}(a_{b})$ are
 comparable with $[T_{f}^{-q_{n-1}} \bar{a}_{b},T_{f}^{q_{n-1}}
\bar{a}_{b}]$. $\tau_0$ is the middle point of the interval $[T_{f}^{-q_{n-1}}\bar{a}_{b},\bar{a}_{b} ]$}
\end{figure}
Set 
\begin{equation}
d_{n}:=
\frac{1}{2} min\left\lbrace  l([\overline{a}_{b},T_{f}^{q_{n-1}}\overline{a}_{b}]),l([T_{f}^{-q_{n-1}}\overline{a}_{b},\overline{a}_{b}]) \right\rbrace. 
\end{equation}
Using this and Corollary \ref{cor2.2} we get
\begin{equation}\label{eq43}
e^{-v}\frac{1}{2}  l([\overline{a}_b,T_f^{q_{n-1}}\overline{a}_b])\leq d_n\leq e^v\frac{1}{2}\,l([\overline{a}_b,T_f^{q_{n-1}}\overline{a}_b]).
\end{equation}

Notice, that the interval $[\tau_{0},T_{f}^{q_{n-1}}\tau_{0}]$  is one of the two generators of the  partition $\xi_{n}(\tau_{0}).$  Hence the intervals 
$T_{f}^{i}[\tau_{0},T_{f}^{q_{n-1}}\tau_{0}],$\,\,$i\in [0,q_{n})$\,\,
cover  the break point  ${a}_{b}$ only once. Using the constants $C_{0}$ and $\zeta_{0}$   in Lemma \ref{lemm4.3}  we define two neighbourhoods(see figure 2)
of the point $\overline{a}_{b}$:
$$V_{n}(\overline{a}_{b})=(\overline{a}_{b}-\frac{1}{2}e^{-v}C_{0}^{-1}d_{n},\,\,\overline{a}_{b}+\frac{1}{2}e^{-v}C_{0}^{-1}d_{n}),$$
$$U_{n}(\overline{a}_{b})= [\overline{a}_{b}-\frac{1}{2}\zeta_{0}l(V_{n}(\overline{a}_{b}),\,\,\overline{a}_{b}+\frac{1}{2}\zeta_{0}l(V_{n}(\overline{a}_{b})].$$
It is clear that $U_{n}(\overline{a}_{b})\subset V_{n}(\overline{a}_{b})\subset [\tau_{0},T_{f}^{q_{n-1}}\tau_{0}].$ 
The   construction of  the intervals $[z_{s},z_{s+1}]$ will depend of the location of 
 $ \overline{c}_{b}$ in the interval $V_{n}(\overline{a}_{b}). $ 
There are two  possibilities to consider:
\begin{eqnarray}\label{eq44}
{either}\,\,\, \overline{c}_{b}\notin U_{n}(\overline{a}_{b}), i.e.\,\, \overline{c}_b\in[T_f^{q_n}\tau_0,T_f^{q_{n-1}}\tau_0]\setminus U_n(\overline{a}_b)\,\,\, or\,\,\overline{c}_{b}\in U_{n}(\overline{a}_{b})).
\end{eqnarray}

Consider the  first case, when  $ \overline {c}_{b}\in V_{n}(\overline{a}_{b})\setminus U_{n}(\overline{a}_{b})).$ 
In this case we set
$$z_{2}=\overline{a}_{b}, z_{3}=\overline{a}_{b}+\frac{1}{4}l(U_{n}(\overline{a}_{b})),z_{4}=\overline{a}_{b}+\frac{1}{2}l(U_{n}(\overline{a}_{b}))\,\,and
\,\,z_{1}=\overline{a}_{b}-\frac{1}{4}l(U_{n}(\overline{a}_{b})).$$

It is easy to see, that the interval  $[z_{1},z_{4}]$ is a subset of $[\tau_{0},T_{f}^{q_{n-1}}\tau_{0}]$ and it does  not contain the break point $\overline{c}_{b}.$
Next we check, that the intervals $[z_{s},z_{s+1}]$,\,\, $s=1,2,3$  satisfy  properties 1)-3) in Lemma \ref{lemm4.4}. The interval $[z_1,z_4]$ is $q_n$-small, because $[z_1,z_4]\subset [\tau_0,T_f^{q{n-1}}\tau_0]$ which is one of the generators of the dynamical partition $\xi_n(\tau_0)$.
By construction, the length of $[z_{1},z_{4}]$ is equal to $\frac{3}{4}e^{-v}C_{0}^{-1}\zeta_{0}d_{n},$ but  $d_{n}$ is half the length   of one of the intervals 
 $[T_{f}^{-q_{n-1}}\overline{a}_{b},\overline{a}_{b}]$ or
$[\overline{a}_{b},T_{f}^{q_{n-1}}\overline{a}_{b}] .$  Consequently, the length of $[z_{1},z_{4}]$ is 
$4e^{v}C_{0}\zeta_{0}^{-1}$- comparable with $[x_{0},T_{f}^{q_{n-1}}x_{0}] $. Next we check, that the assumptions of Lemma \ref{lemm4.1} hold for both intervals $[z_{s},z_{s+1}]$ and  $[T_{f}^{q_{n}}z_{s},T_{f}^{q_{n}}z_{s+1}]$. Note first, that  the lengths of the intervals $[z_{s},z_{s+1}]$, $s=1,2,3 $ are  equal to $\frac{1}{4}e^{-v}C_{0}^{-1}\zeta_{0}d_{n},$ and 
that the  intervals $[z_{s},z_{s+1}]$ and 
$T_{f}^{q_{n-1}}[z_{s},z_{s+1}]$ are $e^{v}$- comparable for $s=1,2,3 .$ Hence  assumption a) of Lemma \ref{lemm4.1} holds true for these intervals with constant $e^{v}.$

Next we check assumption b) of Lemma \ref{lemm4.1}. It is easy to see that for all $i=1,2,3,4$
\begin{eqnarray}\label{eq45}
|x_{0}-z_{i}|\le |x_{0}-z_{2}|+|z_{4}-z_{1}|=|x_{0}-z_{2}|+d_{n},\nonumber\\ 
|x_{0}-T_{f}^{q_{n}}z_{i}|\le |x_{0}-z_{2}|+|z_{2}-T_{f}^{q_{n}}z_{2}|+|T_{f}^{q_{n}}z_{4}-T_{f}^{q_{n}}z_{1}|.
\end{eqnarray}
The point $z_{2}$ belongs to $[T_{f}^{q_{n}}x_{0},T_{f}^{q_{n-1}}x_{0}]\subset[T_{f}^{-q_{n-1}}x_{0},T_{f}^{q_{n-1}}x_{0}]$, 
which is $e^{v}$ comparable with   $[T_{f}^{-q_{n-1}}\overline{a}_{b},T_{f}^{q_{n-1}}\overline{a}_{b}]).$ 
But the length of this last interval is $4e^{v}$-comparable with $d_{n}.$ In complete analogy we can estimate  the second expression in (\ref{eq45}) by  $12e^{v}C_{0}\zeta_{0}^{-1}d_{n}$.\\ 
By assumption, the intervals $T_{f}^{i}[z_{1},z_{1}]$ do  not cover the break point $c_{b},$ but they
cover  the  point $a_{b}$ exactly  once with $z_{2}=T_{f}^{-l}a_{b}.$

Next consider the second case, when  $ \overline {c}_{b}\in U_{n}(\overline{a}_{b}).$ 
 There are again two possibilities:
If \,\, $ \overline {c}_{b}\in [\overline{a}_{b}-\frac{1}{2}\zeta_{0}l(V_{n}(\overline{a}_{b})),\,\overline{a}_{b}],$
 we set
$$z_{1}=\overline{a}_{b}-\frac{1}{2}l(V_{n}(\overline{a}_{b})),z_{2}=\overline{a}_{b}, z_{3}=\overline{a}_{b}+\frac{1}{2}C_{0}l(V_{n}(\overline{a}_{b})),z_{4}=\overline{a}_{b}+C_{0}l(V_{n}(\overline{a}_{b})).
$$
Then  the length of the first interval $[z_{1},z_{2}]$ is  
$\frac{1}{2}e^{-v}C_{0}^{-1}d_{n}$ and of the  other ones it is
  equal to $\frac{1}{2}e^{-v}d_{n}$.
Hence   the lengths of these intervals are $2e^{v}C_{0}$-comparable with $d_{n}$. 
The first two statements of Lemma \ref{lemm4.4} for these intervals can be checked in complete analogy to the first case in (\ref{eq44}).

Next we show, that in the present case the intervals\\ $\left\lbrace T_{f}^{i}[z_{1},z_{2}],\,\,T_{f}^{i}[z_{2},z_{3}],\,\, T_{f}^{i}[z_{3},z_{4}],\,\, 0\leq i \leq q_{n} \right\rbrace $\,  cover both break 
points $a_{b}$, $c_{b}$ regularly with constants $C_{0}$ and $\zeta_{0}$.
By construction, these intervals cover both break points exactly once. Moreover we have
 $z_{2}=\overline{a}_{b}$ and $ \overline {c}_{b}\in [z_{1},z_{2}].$
It is easy to see, that 
$\xi(0)=\frac{z_{3}-z_{2}}{z_{2}-z_{1}}=C_{0}.$ Since $ \overline {c}_{b}\in [\overline{a}_{b}-\frac{1}{2}\zeta_{0}l(V_{n}(\tau_{0})),\,\overline{a}_{b}]$,      we find that
$z(0)=\frac{z_{2}-\overline {c}_{b}}{z_{2}-z_{1}}\le \zeta_{0}.$
So the intervals $[z_{s},z_{s+1}]$,\,\,  $s=1,2,3$ satisfy the statements of Lemma \ref{lemm4.4}.

Consider finally the case, when the break point $\overline{c}_{b}$ is  in the interval $ [ \overline{a}_{b},\overline{a}_{b}+\frac{1}{2}\zeta_{0}l(V_{n}(\tau_{0})]$.
In this case we set 
$$z_{1}=\overline{a}_{b}-C_{0}l(V_{n}(\overline{a}_{b})), z_{2}=\overline{a}_{b}-\frac{1}{2}C_{0}l(V_{n}(\overline{a}_{b})), z_{3}=\overline{a}_{b},  z_{4}=\overline{a}_{b}+\frac{1}{2}l(V_{n}(\overline{a}_{b})). 
$$
The proof of Lemma \ref{lemm4.4} for these intervals $[z_{s},z_{s+1}]$,\, $s=1,2,3$ 
proceeds now exactly as in the  previous case. This concludes the proof of
Lemma \ref{lemm4.4}.\\

\textbf{Proof of Lemma \ref{lemm4.5}}. Assume, that the circle homeomorphism $T_{f}$ satisfies the
assumptions of Theorem \ref{DLM} and the intervals $ [z_{s},z_{s+1}]$, $s=1,2,3$ satisfy  Lemma \ref{lemm4.4}. 
Consider first the case when the intervals 
$\left\lbrace T_{f}^{i}[z_{1},z_{2}],\,\,T_{f}^{i}[z_{2},z_{3}]\,
 T_{f}^{i}[z_{3},z_{4}],\,\,\,0\leq i \leq q_{n}-1 \right\rbrace $\, 
 cover both break 
points $a_{b}$, $c_{b}$ regularly with constants $C_{0}$ and
$\zeta_{0}.$ Suppose that
$z_{2}=\overline{a}_{b}=T_{f}^{-l}{a}_{b}  $ and  $\overline{c}_{b}=T_{f}^{-p}{c}_{b},$ for some $0\leq l,p\leq q_{n}$
and
\begin{eqnarray}\label{eq46}
\frac{z_{3}-z_{2}}{z_{2}-z_{1}}\le C_{0},\,\,\ 0\le \frac{z_{3}-\overline{c}_{b}}{z_{2}-z_{1}}\le \zeta_{0}.\,\,
\end{eqnarray}
Lemma \ref{lemm4.2} shows that
 \begin{eqnarray}\label{eq47}
 Dist(z_{1},z_{2},z_{3},z_{4};T_{f}^{q_{n}})
=[G(\xi(l)+\chi_{1}][F(\xi(p),z(p))+\chi_{2}]\times \nonumber\\
\times\prod_{\substack{0\leq i<q_{n} \\ i\neq l,p }}Dist (T_{f}^{i}z_{1},T_{f}^{i}z_{2}, T_{f}^{i}z_{3},T_{f}^{i}z_{4};T_{f})
\end{eqnarray}
if  $z_{2}=\overline{a}_{b},$ and $\overline{c}_{b}\in \left[z_{1}, z_{2} \right] $  where
\begin{eqnarray}\label{eq48}
|\chi_{j}|=|\chi_{j}(z_{1},z_{2},z_{3},z_{4})|\leq
K_{1} \overset{z_4}{\underset{z_1}{\int}}|D^{2}f(y)|dy,\,\, j=1,2.
\end{eqnarray}
with some constant $K_{1}$ not depending on $n$ and $\varepsilon.$
Next we estimate the right hand side in equation \ref{eq47}. Fix some $\varepsilon>0.$ By assumption, the second  derivative $D^{2}f$ of the  lift $f$ belongs to $L^{1}(S^{1},dl).$ Hence it can be written 
in the form $D^{2}f(x)=g_{\varepsilon}(x)+\theta_{\varepsilon}(x)$ with
$g_{\varepsilon}$  a continuous function on $S^{1}$ and 
$\parallel\theta_{\varepsilon}(x)\parallel_{L^{1}}< \varepsilon.$
By  assumption, among the intervals $T_{f}^{i}[z_{s},z_{s+1}]$, $0\leq i\leq q_{n}$, only the intervals 
$T_{f}^{l}[z_{1},z_{4}]$ and $T_{f}^{p}[z_{1},z_{4}]$ contain the  break points ${a}_{b}$ respectively ${c}_{b}$.

Obviously 
\begin{displaymath}
|\prod_{\substack{0\leq i<q_{n} \\ i\neq l,p }}Dist(T_{f}^{i}z_{1},T_{f}^{i}z_{2},T_{f}^{i}z_{3},T_{f}^{i}z_{4};
T_{f}) -1|= 
\end{displaymath}

\begin{eqnarray}\label{eq49}
=|exp\{\overset{q_{n}-1}{\underset{i=0,\;i\neq l,b}
{\sum}}
\text{log}(1+(Dist(T_{f}^{i}z_{1},T_{f}^{i}z_{2},T_{f}^{i}z_{3},T_{f}^{i}z_{4};
T_{f})-1))\}-1|
\end{eqnarray}

Next applying Theorem \ref{theo3.1} we obtain 
\begin{eqnarray}\label{eq50}
|Dist(T_{f}^{i}z_{1},T_{f}^{i}z_{2},T_{f}^{i}z_{3},T_{f}^{i}z_{4};
T_{f})-1|\le C_{1}|T_{f}^{i}z_{4}-T_{f}^{i}z_{1}|\times\nonumber\\
\overset{}{\underset{x,t\in
[T_{f}^{i}z_{1},T_{f}^{i}z_{4}]}
{\max}|g_{\varepsilon}(x)-g_{\varepsilon}(t)|}
+C_{1}\overset{T_{f}^{i}z_{4}}{\underset{T_{f}^{i}z_{1}}{\int}}|\theta_{\varepsilon}(y)|dy+C_{1}
\Big(\overset{T_{f}^{i}z_{4}}{\underset{T_{f}^{i}z_{1}}{\int}}|D^{2}f(y)|dy\Big)^{2}
\end{eqnarray}
where the constant $C_{1}>0$ depends only on the function $f.$

For  $D^{2}f\in L^{1}([0,1])$  the function
$
\Psi(x)=\overset{x}{\underset{0}{\int}}<D^{2}f(y)|dy
$
is absolutely  continuous on $S^{1}$.  Note that
the  functions $\Psi(x)$
and  $g_{\varepsilon}(x)$
are uniformly continuous on $ S^{1}$ because they are continuous on $S^{1}$. 
Hence there exists
$\delta_{0}=\delta_{0}(\varepsilon)>0$ such that for any $x,t\in S^{1}$ with $|x-t|<\delta_{0},$
the inequalities 
\begin{eqnarray}\label{eq51}
|\Psi(x)-\Psi(t)|<\varepsilon,\,\,
|g_{\varepsilon}(x)-g_{\varepsilon}(t)|<\varepsilon,
\end{eqnarray}
hold true.\\
 By assumption, the interval  $[z_{1},z_{4}]$ is $q_{n}-$small. Hence by
Corollary \ref{cor2.2} for all $   0\leq i\leq q_{n}$ we have
$
l(T_{f}^{i}[z_{1},z_{4}])\leq const\,\lambda^{n},\ 0 \leq i
\leq q_n-1,
$\,\,
with $\lambda=(1+e^{-v})^{\frac{-1}{2}}<1.$ Consequently there exists a number
$N_{0}=N_{0}(\delta)>0$ such that for  $n>N_{0}$ and all $0 \leq i\leq q_{n}$ one has 
$
l(T_{f}^{i}[z_{1},z_{4}])\leq \delta_{0}.
$ 
This together with (\ref{eq51}) implies  that for   $n>N_{0}$  the following
inequalities
\begin{eqnarray}\label{eq52}
|\Psi(x)-\Psi(y)|<\varepsilon,\,\,\,\,
|g_{\varepsilon}(x)-g_{\varepsilon}(y)|<\varepsilon,
\end{eqnarray}
hold for  all $x,y\in [T_{f}^{i}z_{1},T_{f}^{i}z_{4}],$ and all $0 \leq i\leq q_{n}$.

On the other hand, since the interval $[z_{1},z_{4}]$ is $q_{n}-$ small, the intervals 
$T_{f}^{i}[z_{1},z_{4}]$, $0 \leq i\leq q_{n}$, 
  are non intersecting and trivially 
\begin{displaymath}
\sum_{i=0}^{q_{n}}l(T_{f}^{i}[z_{1},z_{4}])\le1.
\end{displaymath}
Since $\parallel\theta_{\varepsilon}\parallel_{L^{1}}< \varepsilon$,
we find, using relations  (\ref{eq49})-(\ref{eq52}), 
\begin{displaymath}
|\prod_{\substack{0\leq i<q_{n} \\ i\neq l,p }}Dist(T_{f}^{i}z_{1},T_{f}^{i}z_{2},T_{f}^{i}z_{3},T_{f}^{i}z_{4};
T_{f})-1| \le
\end{displaymath}
 \begin{displaymath}\le 
C_{2}\sum_{i=0}^{q_{n}-1}
|Dist(T_{f}^{i}z_{1},T_{f}^{i}z_{2},T_{f}^{i}z_{3},T_{f}^{i}z_{4};
T_{f})-1|\le C_{2}\varepsilon \sum_{i=0}^{q_{n}-1}
|T_{f}^{i}z_{4}-T_{f}^{i}z_{1}|+
\end{displaymath}\begin{displaymath}
+C_{2} \sum_{i=0}^{q_{n}-1}\overset{T_{f}^{i}(z_4)}
{\underset{T_{f}^{i}(z_1)}{\int}}|\theta_{\varepsilon}(y)|dy 
+C_{2} \sum_{i=0}^{q_{n}-1}
|\Psi(T_{f}^{i}z_{4})-\Psi(T_{f}^{i}z_{1})|\overset{T_{f}^{i}(z_4)}
{\underset{T_{f}^{i}(z_1)}{\int}}|D^{2}f(y)|dy\le 
\end{displaymath}
\begin{eqnarray}\label{eq53}
\,\,\,\,\,\,\,\,\le  C_{2}\left\{2\varepsilon+\overset{1}
{\underset{0}{\int}}|\theta_{\varepsilon}(y)|dy+\varepsilon\overset{1}
{\underset{0}{\int}}|D^{2}f_{1}(y)|dy\right\}
\le  C_{2}(3+||D^{2}f||_{L^{1}})\varepsilon 
\end{eqnarray}
where the constant depends only on $f.$.
 
Next we estimate the expression
$(G(\xi(l))+\chi_{1})(F(\xi(p),z(p))+\chi_{2})$.
where  $|\chi_{i}|$ is bounded above by
$K_{1} \overset{z_4}{\underset{z_1}{\int}}|D^{2}f(y)|dy$,\,\, for $j=1,2$ 
with some constant $K_{1}$ not depending on $n$ and $\varepsilon$. For $n>N_{0}$ we have
 $\overset{z_4}{\underset{z_1}{\int}}|D^{2}f(y)|dy=\Psi(z_{4})-\Psi(z_{1})<\varepsilon.$
Since  $G(x)$ \,\, is bounded for $x>0$  and $F(x,t),$ is bounded for $x>0$ and $1\leq t \leq 1$,
it is hence sufficient to estimate the term $G(\xi(l))F(\xi(p),z(p))$ in the above product.
By assumption, the intervals
$\left\lbrace T_{f}^{i}[z_{1},z_{2}],\,\,T_{f}^{i}[z_{2},z_{3}]\,
 T_{f}^{i}[z_{3},z_{4}],\,\,\,0\leq i \leq q_{n}-1 \right\rbrace $\, 
 cover both break 
points $a_{b}$, $c_{b}$ regularly with constants $C_{0}$ and
  $\zeta_{0}.$  Applying Lemma \ref{lemm4.3} we conclude
$$
|G(\xi(l))F(\xi(p),z(p))-1|\geq \frac{|\sigma(a_{b})\sigma(c_{b})-1|}{4}>0,
$$
since by assumption the product of the jump ratios of $Df$ at the break points is nontrivial i.e.   $\sigma(a_{b})\sigma(c_{b})\neq 1$. It then follows, that  for sufficiently large $n$ and small $\varepsilon$ the inequality
$$|[G(\xi(l))+\chi_{1}][F(\xi(p),z(p))+\chi_{2}]-1|\geq \frac{|\sigma(a_{b})\sigma(c_{b})-1|}{8}$$ 
holds true. This together with (\ref{eq44}) and (\ref{eq53}) imply 
the assertion of Lemma \ref{lemm4.5} in the  case of a regular covering of the two break points.\\
W.l.o.g. we assume next that the intervals $\left\lbrace T_{f}^{i}[z_{1},z_{2}],\,\,T_{f}^{i}[z_{2},z_{3}]\,
 T_{f}^{i}[z_{3},z_{4}],\,\,\,0\leq i \leq q_{n}-1 \right\rbrace $\, 
cover only the break point $a_{b}$ with $z_{2}=\overline{a}_{b}=T_{f}^{-l}a_{b}$ 
for some $0\leq l\leq q_{n} $
and  satisfy properties 1), 2)  of Lemma \ref{lemm4.4}.

We write $Dist(z_{1},z_{2},z_{3},z_{4};T_{f}^{q_{n}})$ in the following form 
 \begin{eqnarray}\label{eq54}
 Dist(z_{1},z_{2},z_{3},z_{4};T_{f}^{q_{n}})
=Dist (T_{f}^{l}z_{1},T_{f}^{l}z_{2}, T_{f}^{l}z_{3},T_{f}^{l}z_{4};T_{f})\times \nonumber\\
\times\prod_{\substack{0\leq i<q_{n} \\ i\neq l }}Dist (T_{f}^{i}z_{1},T_{f}^{i}z_{2}, T_{f}^{i}z_{3},T_{f}^{i}z_{4};T_{f}).
\end{eqnarray}
For sufficiently large $n$ and any $\varepsilon >0$ the product over $i\neq l, p$ in (\ref{eq54}) takes it value in an
$\varepsilon- $neighbourhood of 1.
By assumption, only the interval $T_{f}^{l}[z_{1},z_{4}]$ contains the break point $a_{b}$ 
with $a_{b}=T_{f}^{l}z_{2}.$ Using  Lemma \ref{lemm3.2} we find
\begin{eqnarray}\label{eq55}
|Dist (T_{f}^{l}z_{1},T_{f}^{l}z_{2}, T_{f}^{l}z_{3},T_{f}^{l}z_{4};T_{f})- \frac{\sigma(a_{b})(1+\xi(l))}{\sigma(a_{b})+\xi(l)}|\nonumber\\
\le  K_1 \overset{z_4}{\underset{z_1}{\int}}|D^{2}f(y)|dy,
\end{eqnarray}
where the constant $K_1>0$ depends only on the function $f$ and where $\xi(l)=\frac{T_{f}^{l}z_{3}-T_{f}^{l}z_{2}}{T_{f}^{l}z_{2}-T_{f}^{l}z_{1}}.$
Obviously
$$\frac{\sigma(a_{b})(1+\xi(l))}{\sigma(a_{b})+\xi(l)}-1=\frac{(\sigma(a_{b})-1)\xi(l)}{\sigma(a_{b})+\xi(l)}.$$
Using this and the inequalities $R_{2}^{-1}\le \xi(l)\le R_{2}$, following from (\ref{eq22}), and the comparability of the intervals  $[z_{s}, z_{s+1}]$ for $s=1,2,3$, we obtain 
 $$R_{3}^{-1} \le \frac{(\sigma(a_{b})-1)\xi(l)}{\sigma(a_{b})+\xi(l)}\le R_{3}$$
where the constants $R_{i}>0,\,\,i=1,2$  depend only on $f.$ Finally we obtain 
\begin{eqnarray}\label{eq56}
|Dist (T_{f}^{l}z_{1},T_{f}^{l}z_{2}, T_{f}^{l}z_{3},T_{f}^{l}z_{4};T_{f})-1|\geq const>0,
\end{eqnarray}
where the constant again depends only on $f$.
 As in the first case, the inequality \\
$$ |\prod_{\substack {0\leq i<q_{n} \\ i\neq l }}Dist (T_{f}^{i}z_{1},T_{f}^{i}z_{2}, T_{f}^{i}z_{3},T_{f}^{i}z_{4};T_{f})-1|\le const\,\, \varepsilon.$$
holds true also in the present case.
This together with (\ref{eq54}) and (\ref{eq56}) proves Lemma \ref{lemm4.5}.

\textbf{Proof of Theorem \ref{DLM1}}. The idea of the  proof of Theorem \ref{DLM1} is completely similar to
the one of Theorem \ref{DLM}. Hence we will only give the construction of the intervals $[z_{s}, z_{s+1}]$,\,\,\, $s=1,2,3$,  which play the key role in the proof.\\
Consider the $n-$th dynamical partition $\xi_{n}(\overline{a}_{b})$ of the preimage $\overline{a}_{b}$ of the break point 
$a_{b}$ in the interval $[T_{f}^{q_{n}}x_{0}, T_{f}^{q_{n}-1}x_{0}]$ around the point $x_{0}$, 
at which there exists a positive derivative $DT_{\varphi}(x_0)$ of the 
conjugating homeomorphism $T_{\varphi}(x).$  Since the rotation number $\rho_{f}$ is irrational of bounded type there exists a subsequence $\left\lbrace n_{k},k=1,2,...\right\rbrace \in\mathbb{N} $ such that for every
$n_{k}$ the interval $[\overline{a}_{b}, T_{f}^{q_{n{k}-1}}\overline{a}_{b}]$ respectively  the interval $[T_{f}^{q_{n{k}}}\overline{a}_{b}, \overline{a}_{b}]$ contains the point 
$\overline{c}_{b}=T_{f}^{-p}c_{b}$ for some $p\in [0,q_{n_{k}})$ and such that furthermore $K_{3}^{-1}\le \frac{\overline{c}_{b}-\overline{a}_{b}}{T_{f}^{q_{n_{k}-1}}\overline{c}_{b}-\overline{a}_{b}}\le K_{3}$
respectively  $K_{3}^{-1}\le \frac{\overline{a}_{b}-\overline{c}_{b}}{\overline{c}_{b}-T_{f}^{q_{n_{k}}}\overline{a}_{b}}\le K_{3}$ holds, where the constant $K_{3}$ depends only on $f.$
Set $d_{n_{k}}=min\left\lbrace|\overline{c}_{b}-\overline{a}_{b}|, |T_{f}^{q_{n_{k}-1}}\overline{a}_{b}-\overline{c}_{b}| \right\rbrace $ and  define
the points 
$$z_{2}=\overline{c}_{b}, z_{1}=\overline{c}_{b}-\frac{1}{2}d_{n_{k}}, z_{3}=\overline{c}_{b}+\frac{1}{2}d_{n_{k}}, z_{4}=\overline{c}_{b}+d_{n_{k}}.$$
As in the proof of Lemma \ref{lemm4.4} it  can be checked that the intervals $[z_{s}, z_{s+1}]$,\,\,
 $s=1,2,3$ then satisfy the statements of this Lemma.\\

{\bf{Acknowledgement}}\\
The work of Akhtam Dzhalilov during a stay at the University of Clausthal was supported by the German Research Council (DFG) under project  Ma 633/18-1.

\end{document}